\documentclass[a4paper,reqno]{amsart}

\usepackage{amsmath}
\usepackage{amssymb}
\usepackage{amsthm}
\usepackage{ifpdf}
\ifpdf
\usepackage[pdftex,plainpages=false,pdfpagelabels,linktocpage,pdfstartview=FitH,colorlinks=true,linkcolor=blue,citecolor=magenta]{hyperref}
\else
\usepackage[hypertex,linkcolor=cyan]{hyperref}
\fi
\usepackage[initials,msc-links]{amsrefs}[2007/10/22]
\theoremstyle{plain}
\newtheorem{thm}{Theorem}[section]
\newtheorem{lemma}[thm]{Lemma}
\newtheorem{prop}[thm]{Proposition}
\newtheorem{corol}[thm]{Corollary}
\theoremstyle{remark}

\newtheorem{rem}[thm]{Remark}
\theoremstyle{definition}
\newtheorem{defn}[thm]{Definition}
\newtheorem*{convention}{Convention}
\newtheorem*{notation}{Notation}

\numberwithin{equation}{section}

\newcommand{\R}{\mathbb{R}}
\newcommand{\C}{\mathbb{C}}
\renewcommand{\H}{\mathbb{H}}
\renewcommand{\P}{\mathbb{P}}

\newcommand{\del}{\partial}

\newcommand{\unC}{\underline{\C}}
\newcommand{\fo}{\mathfrak{o}}
\newcommand{\fsl}{\mathfrak{sl}}
\newcommand{\fgl}{\mathfrak{gl}}
\newcommand{\rSL}{\mathrm{SL}}
\newcommand{\rGL}{\mathrm{GL}}
\newcommand{\rO}{\mathrm{O}}
\newcommand{\rSO}{\mathrm{SO}}
\newcommand{\dom}{\mathrm{dom}}
\newcommand{\cL}{\mathcal{L}}
\newcommand{\cD}{\mathcal{D}}
\newcommand{\cN}{\mathcal{N}}
\newcommand{\cA}{\mathcal{A}}
\newcommand{\cQ}{\mathcal{Q}}
\newcommand{\Ad}{\mathrm{Ad}}
\newcommand{\ad}{\mathrm{ad}}
\newcommand{\Res}{\mathrm{Res}}
\newcommand{\hd}{\hat{d}}

\newcommand{\hA}{\hat{A}}
\newcommand{\hV}{\hat{V}}
\newcommand{\hL}{\hat{\Lambda}}
\newcommand{\hq}{\hat{q}}
\newcommand{\hS}{\hat{S}}
\newcommand{\set}[1]{\{#1\}}
\newcommand{\End}{\mathrm{End}}
\newcommand{\trace}{\mathrm{trace}}
\newcommand{\II}{\mathrm{I\kern-1ptI}}
\newcommand{\half}{\tfrac12}   
\DeclareSymbolFont{script}{U}{eus}{m}{n}
\DeclareSymbolFontAlphabet{\mathscr}{script}
\DeclareMathSymbol{\EuWedge}{0}{script}{"5E}
\newcommand{\Wedge}{\EuWedge}

\begin{document}

\title{Dressing transformations of constrained Willmore surfaces}

\author[F.E. Burstall]{Francis E. Burstall}
\address{Department of Mathematical Sciences\\ University of Bath\\
Bath BA2 7AY\\UK} \email{feb@maths.bath.ac.uk}
\author[A.C. Quintino]{\'Aurea C. Quintino}
\address{Centro de Matem\'{a}tica e Aplica\c{c}\~{o}es Fundamentais da Universidade de Lisboa\\
Avenida Professor Gama Pinto\\ 2\\ 1649-003 Lisboa\\ Portugal}
\email{aurea@ptmat.fc.ul.pt}


\begin{abstract}
We use the dressing method to construct transformations of
constrained Willmore surfaces in arbitrary codimension.  An adaptation
of the Terng--Uhlenbeck theory of dressing by simple factors to this
context leads us to define B\"acklund transforms of these surfaces
for which we prove Bianchi permutability.  Specialising to
codimension $2$, we generalise the Darboux transforms of Willmore
surfaces via Riccati equations, due to
Burstall--Ferus--Leschke--Pedit--Pinkall, to the constrained Willmore
case and show that they amount to our B\"acklund transforms with real
spectral parameter.
\end{abstract}
\maketitle
\section*{Introduction}

The Willmore functional is a conformally invariant functional,
defined on immersed surfaces in the conformal $n$-sphere, with a
long, interesting history \cite{Bry84,Bla29,Ger31,Tho23,Wil65}.  Its
extrema are the \emph{Willmore surfaces} and these have attracted a
lot of attention in recent years, in large part due to interest in
the celebrated Willmore conjecture, now affirmed by Marques--Neves
\cite{MarNev12}.

One approach to Willmore surfaces is via a certain Gauss map: to any
surface in $S^n$, one may associate \cite{Bla29,Bry84} its
\emph{central sphere congruence}.  Geometrically, this is a family of
$2$-spheres tangent to the surface and sharing mean curvature vectors
with it.  Alternatively, it may be viewed as a map into the space of
$2$-spheres which is a certain Grassmannian of $4$-planes
\cite{Roz48}.  A key result is that a surface is Willmore if and only
if its central sphere congruence is a harmonic map
\cite{Bla29,Bry84,Eji88,Rig87} and then the well-developed theory of
harmonic maps applies.  In particular, the machinery of integrable
systems theory becomes available.

The starting point of the integrable systems approach to harmonic
maps is the observation \cite{Poh76,Uhl89,ZakMik78,ZahSab79} that the
harmonic map equations amount to the flatness of a family of flat
connections depending on a \emph{spectral parameter}.  This structure
gives rise to symmetries, of which we shall say much more below, and
constructions via completely integrable Hamiltonian systems and
algebraic curves \cite{BurFerPedPin93,Hit90}.

This zero curvature representation of the harmonic map equations
allows one to deduce two kinds of symmetry.  Firstly, harmonic maps
admit a \emph{spectral deformation} \cite{Uhl89} by exploiting a
scaling freedom in the spectral parameter.  Secondly, harmonic maps
admit \emph{dressing transformations} \cite{TerUhl00,Uhl89} which
arise by applying carefully chosen gauge transformations to the
family of flat connections.

A larger class of surfaces arises when one imposes the weaker
requirement that a surface extremise the Willmore functional only
with respect to variations which preserve the conformal structure:
these are the \emph{constrained Willmore surfaces}.  Now the
Euler--Lagrange equations include a Lagrange multiplier in the form
of a holomorphic quadratic differential \cite{BohPetPin08} and, while
the central sphere congruence is no longer harmonic, a loop of flat
connections is still available \cite{Boh10,BurPedPin02}.

It is the purpose of the present paper to apply the theory of
dressing transformations to constrained Willmore surfaces in $S^{n}$.
Two considerations prevent this from being a routine exercise in the
existing theory for harmonic maps: first, that theory applies to the
central sphere congruence qua map into a Grassmannian and it requires
some care to see that the transformed map is the central sphere
congruence of a surface at all, let alone a constrained Willmore one.
Moreover, the theory must be extended somewhat to cover the possibly
non-harmonic central sphere congruences of constrained Willmore
surfaces.  

To address the second problem, we introduce and study the notion of a
\emph{$k$-perturbed harmonic map} which includes (with $k=2$) the
central sphere congruences of constrained Willmore surfaces.  The
idea is that the flat connections associated to harmonic maps are
characterised by having simple poles (with respect to the spectral
parameter) at $0$ and $\infty$.  Relaxing this requirement to allow
poles of order $k$ leads us to $k$-perturbed harmonic maps.  These
maps are a geometric incarnation of the $k$-th elliptic integrable
system associated with a symmetric space in the sense of Terng
\cite{Ter08} (see also Khemar \cite{Khe12}) and also fit into the
framework of Brander--Dorfmeister \cite{BraDor09}.  As such,
$k$-perturbed harmonic maps make sense with any symmetric space as
target but we focus on the space of $2$-spheres that is of immediate
interest to us.  We define spectral deformations and dressing
transformations of such maps, the latter mediated by gauge
transformations we call \emph{dressing gauges}.

We define an energy density for $k$-perturbed harmonic maps which
coincides with the harmonic map energy density when $k=1$ and prove
that this changes by an explicit exact term under our dressing
transformations.  This result, of possible independent interest,
appears to be new, even for harmonic maps.

We then show that our dressing transformations preserve the class of
central sphere congruences of constrained Willmore surfaces.  In this
setting, we see that the Willmore density changes by an exact term
only and, moreover, the  holomorphic quadratic differential that acts as
Lagrange multiplier is fixed by the dressing transformations.

We complete our general theory by adapting the Terng--Uhlenbeck
\cite{TerUhl00} theory of dressing by simple factors to our context.
Thus we construct dressing gauges algebraically from parallel
subbundles of one of the given flat connections.  We call the
resulting dressing transformations \emph{B\"acklund transformations}
and prove a Bianchi permutability theorem for them.

Finally, we restrict attention to the case of surfaces in $S^4$.
Here, there is a well-developed quaternionic formalism
\cite{Boh10,BohPedPin09a,BurFerLesPed02,FerLesPedPin01} that can be
brought to bear.  In particular, arguing by analogy with Darboux
transforms of constant mean curvature surfaces, a Darboux transform of
Willmore surfaces in $S^4$ is presented in \cite{BurFerLesPed02}
which proceeds by solving a Riccati equation.  For us, the key point
about this codimension $2$ setting is that the space of (oriented)
$2$-spheres is a (pseudo-)Hermitian symmetric space.  Consequently,
one may apply a canonical gauge transformation, arising from the
ambient complex structure, to replace the family of flat connections
with a simpler \emph{untwisted} family with poles of about half the
order of the original family.  Working with the untwisted family, we
find that the B\"acklund transformations specialise precisely to
simple factor dressing in the sense of Terng--Uhlenbeck and,
moreover, can be derived from solutions of a Riccati equation.  In
particular, we simultaneously extend the Darboux transforms of
\cite{BurFerLesPed02} to the constrained Willmore case and identify
them with the B\"acklund transforms with real spectral parameter.
These results generalise a similar analysis of harmonic maps into the
$2$-sphere and the constant mean curvature surfaces of which they are
Gauss maps that is carried out in \cite{BurDorLesQui13}.

The results of the paper are based, in part, on those in the second
author's PhD thesis \cite{Qui09,Quia} and some of them were
announced in \cite{Qui11}.

\subsection*{Acknowledgements}

The first-named author gratefully acknowledges instructive
conversations with David Calderbank, Katrin Leschke and Udo
Hertrich-Jeromin.  The second-named author was supported by
Funda\c{c}\~{a}o para a Ci\^{e}ncia e a Tecnologia, Portugal, with
the PhD scholarship SFRH/BD/6356/2001, and by Funda\c{c}\~{a}o da
Universidade de Lisboa, with the postdoctoral scholarship
CMAF-BPD-01/09.

\section{Constrained Willmore surfaces and perturbed harmonicity}

\subsection{Conformal submanifold geometry}

\subsubsection{Conformal geometry of the sphere} 

We are going to study submanifolds of the conformal $n$-sphere $S^n$
and find a convenient conformally invariant viewpoint on the latter
in Darboux's light-cone model \cite[Chapitre~VI]{Dar87}.  For this,
contemplate the Lorentzian space-time $\R^{n+1,1}$ with inner product
$(\,,\,)$ and light-cone $\mathcal{L}$.  The corresponding quadric
$\P(\mathcal{L})$ is a conformal manifold: each section $\sigma$ of
$\mathcal{L}^{\times}\to\P(\mathcal{L})$ provides a metric
$g_{\sigma}$ via $g_{\sigma}(X,Y)=(d_X\sigma,d_Y\sigma)$ and
$g_{e^{u}\sigma}=e^{2u}g_{\sigma}$.  Fixing unit-timelike
$t_0\in\R^{n+1,1}$ yields a section $\sigma$ with $(\sigma,t_0)\equiv
-1$ and so an isometry
$\sigma:(\P(\mathcal{L}),g_{\sigma})\cong\{v\in
\mathcal{L}:(v,t_0)=-1\}$ which last is isometric (via $v\mapsto
v-t_0$) to the unit sphere in $\langle t_0\rangle^{\perp}$.  Thus
$\P(\mathcal{L})\cong S^n$ \emph{qua} conformal manifolds.

Moreover, it is clear that the orthogonal group $\rO(n+1,1)$ acts
transitively and conformally on $\P(\mathcal{L})$. In this way, it
actually double covers the M\"obius group of global conformal
diffeomorphisms on $S^n$.

For more details, see \cite[Chapter~1]{Her03}.

In the sequel, we will have much to do with $\rO(n+1,1)$-connections
and so forms with values in the Lie algebra $\fo(n+1,1)$.  We shall
make repeated (and silent!) use of the isomorphism
$\wedge^{2}\R^{n+1,1}\cong o(\R^{n+1,1})$ given by $u\wedge
v(w):=(u,w)v-(v,w)u$, for $u,v,w\in\R^{n+1,1}$.

\subsubsection{Conformally immersed surfaces and the central sphere
congruence} 
\label{sec:conf-immers-surf}
A map $f:\Sigma\to\P(W)$ of a manifold to a projective space is the
same as a line subbundle of the trivial bundle
$\underline{W}=\Sigma\times W$: we identify $f$ with the bundle whose
fibre at $x$ is $f(x)$.  In particular, we will usually view a map
$\Lambda:\Sigma\to \P(\mathcal{L})$ as a null line subbundle of
$\underline{\R}^{n+1,1}$.  From this point of view, sections of
$\Lambda$ are simply lifts of $\Lambda$ to maps $\Sigma\to\R^{n+1,1}$.

Given such a $\Lambda$, we define the derived bundle
\begin{equation*}
\Lambda^{(1)}:=\langle\sigma,d\sigma(T\Sigma)\rangle\leq\Lambda^{\perp},
\end{equation*}
for $\sigma$ an arbitrary lift of $\Lambda$. Note that $\Lambda$ is
an immersion if and only if the bundle $\Lambda^{(1)}$ has rank
$\dim\Sigma+1$ and then every lift is also an immersion
$\sigma:\Sigma\to\R^{n+1,1}$.

Henceforth, we will take $\Sigma$ to be a Riemann surface, thus an
oriented, conformal manifold of real dimension $2$ or, equivalently,
a complex manifold of dimension $1$.  We denote the complex structure
of $\Sigma$ by $J^{\Sigma}$.  The type decomposition of $T\Sigma$
now induces a decomposition of $\Lambda^{1}/\Lambda$: set
\begin{equation*}
\Lambda^{1,0}=\Lambda\oplus d\sigma(T^{1,0}\Sigma)\qquad
\Lambda^{0,1}=\Lambda\oplus d\sigma(T^{0,1}\Sigma),
\end{equation*}
for some, hence any, lift $\sigma$.  Then
\begin{align*}
\Lambda^{1,0}+\Lambda^{0,1}&=\Lambda^{(1)}\\
\Lambda^{1,0}\cap\Lambda^{0,1}&=\Lambda\\
\Lambda^{1,0}&=\overline{\Lambda^{0,1}}
\end{align*}
\begin{convention}
Here and below, we do not distinguish notationally between a
subbundle and its complexification.  In fact, we shall usually take
all subbundles to be complex subbundles of
$\unC^{n+2}=(\underline{\R}^{n+1,1})^{\C}$ and view such a subbundle as
real if it is stable under the complex conjugation on $\unC^{n+2}$
with fixed set $\underline{\R}^{n+1,1}$.
\end{convention}

An immersion $\Lambda$ is conformal if and only if each lift is a
conformal immersion $\Sigma\to\R^{n+1,1}$ if and only if
$\Lambda^{1,0}$ (and hence also $\Lambda^{0,1}$) is isotropic with
respect to the (complex bilinear extension of) $(\,,\,)$.

A key tool in conformal submanifold geometry is the \emph{central
sphere congruence} or \emph{conformal Gauss map}
\cite{Bla29,Bry84,Eji88,Ger31,Rig87} $V$ of a conformal immersion
$\Lambda$.  For us, this is the bundle $V\leq\underline{\R}^{n+1,1}$
of $(3,1)$-planes given by
\begin{equation*}
\Lambda^{(1)}\oplus\langle\Delta\sigma\rangle,
\end{equation*}
for $\sigma$ any lift of $\Lambda$ and $\Delta$ the Laplacian of
$g_{\sigma}$.

From $V$, we get an orthogonal decomposition of the trivial bundle
\begin{equation*}
\underline{\R}^{n+1,1}=V\oplus V^{\perp}
\end{equation*}
and a corresponding reduction of the trivial connection:
\begin{equation*}
d=\cD_V+\cN_V,
\end{equation*}
where $V,V^{\perp}$ are $\cD_V$-parallel and $\cN_V$ takes values in
$V\wedge V^{\perp}\subset \fo(\underline{\R}^{n+1,1})$.
Alternatively, viewing $V$ as a map $\Sigma\to\mathrm{Gr}_{(3,1)}(\mathbb
R^{n+1,1})$ to the Grassmannian of $(3,1)$-planes, $\cN_{V}$ may be
canonically identified with $dV$.

We note that since $\Lambda^{(1)}\leq V$, $\cN_V(\Lambda)=0$ and,
similarly, $\Delta\sigma\in\Gamma V$ is equivalent to
$\cN_V^{1,0}(\Lambda^{0,1})=0$.

\subsection{Constrained Willmore surfaces and holomorphic quadric differentials}
\label{sec:constr-willm-surf}

Let $\Lambda$ be a conformal immersion of a compact Riemann surface
$\Sigma$ with central sphere congruence $V$.  The \emph{Willmore
energy} $W(\Lambda)$ of $\Lambda$ is given by
\begin{equation*}
W(\Lambda)=\half\int_{\Sigma}|\II_0|^2,
\end{equation*}
where $\II_0$ is the trace-free second fundamental form of $\Lambda$
(computed against any representative metric on $S^n$ and independent
of that choice).  This is known
\cite{Bla29,Bry84,BurFerLesPed02,BurCal10,Eji88} to coincide with the
harmonic map energy of $V$:
\begin{equation}\label{eq:79}
W(\Lambda)=E(V)=\half\int_{\Sigma}(\cN_V\circ J^{\Sigma}\wedge\cN_V).
\end{equation}

A conformal immersion is \emph{Willmore} if it extremises the Willmore
functional and \emph{constrained Willmore} if it extremises the
Willmore functional with respect to infinitesimally conformal
variations.  It follows at once from \eqref{eq:79} that a surface is
Willmore when its central sphere congruence is harmonic and it is
well known \cite{Bla29,Bry84,Rig87} that the converse holds.  Thus a
surface is Willmore if and only if $d^{\mathcal{D}}*\mathcal{N}=0$.
More generally, we have the following reformulation of a theorem of
Bohle--Peters--Pinkall \cite{BohPetPin08}:
\begin{thm}[{\cite[Section~14]{BurCal10}}] \label{CWtemp}
$\Lambda$ is a constrained Willmore surface if and only if there
exists a real form $q\in\Omega ^{1}(\Lambda\wedge\Lambda ^{(1)})$
with
\begin{subequations}
\label{eq:14}
\begin{equation}\label{eq:holoQ}
d^{\mathcal{D}_V}q=0
\end{equation}
such that
\begin{equation}\label{eq:89}
d^{\mathcal{D}_{V}}*\mathcal{N}_{V}=2[q\wedge *\mathcal{N}_{V}].
\end{equation}
\end{subequations}
Such a form $q$ is said to be a \emph{(Lagrange) multiplier} for
$\Lambda$.  Thus Willmore surfaces are constrained
Willmore surfaces with zero multiplier.
\end{thm}
When $\Sigma$ is non-compact, we define constrained Willmore surfaces
to be the solutions of \eqref{eq:14}.

The $1$-form $q$ is a bundle-valued avatar of a more familiar
geometric object: a holomorphic quadratic differential on $\Sigma$.
Indeed, given $q\in\Omega^{1}(\Lambda\wedge\Lambda ^{(1)})$, we define a
$2$-tensor $Q$ on $M$ by setting
\begin{equation}\label{eq:Q}
Q(X,Y)\sigma:=q_{X}(d_Y\sigma),
\end{equation}
for $\sigma\in\Gamma\Lambda$. We then have
\begin{lemma}[{\cite[Section~12]{BurCal10}}]\label{th:34}
Let $q\in\Omega^{1}(\Lambda\wedge\Lambda ^{(1)})$ and define $Q$ by
\eqref{eq:Q}.  Then $d^{\cD_{V}}q=0$ if and only if $Q$ is symmetric and
trace-free, thus $Q^{1,1}=0$, and $Q^{2,0}$ is holomorphic.

In this case, $q^{1,0}\in\Omega^{1}(\Lambda\wedge\Lambda^{0,1})$.
\end{lemma}
It is the holomorphic quadratic differential $Q$ that appears in the
Bohle--Peters--Pinkall version of \eqref{eq:89}.

\begin{rem}\label{th:35}
In general, for a given constrained Willmore surface $\Lambda$, the multiplier
$q$ is unique.  The exception is when $\Lambda$ is in addition
isothermic: if $d^{\mathcal{D}_{V}}*\mathcal{N}_V=2[q_{1}\wedge
*\mathcal{N}_V]=2[q_2\wedge *\cN_{V}]$ with $d^{\cD_{V}}q_i=0$ then
$\eta:=*q_2-*q_1$ solves $d^{\cD_{V}}\eta=[\cN_V\wedge\eta]=0$ or,
equivalently, $d\eta=0$.  This last condition amounts to the
existence of a holomorphic quadratic differential that commutes with
the trace-free second fundamental form, that is, that $\Lambda$ is
isothermic.  In this situation, $q_1+*s\eta$ is a multiplier for
$\Lambda$ for each $s\in\R$.

The isothermic Willmore surfaces in $S^3$ are precisely the
minimal surfaces with respect to a constant curvature metric
\cite{Tho23} while the isothermic constrained Willmore surfaces
include those of constant mean curvature \cite{Ric97}.  See
\cite{BohPetPin08,BurPedPin02} for more details on this topic.
\end{rem}

\subsection{Flat connections and perturbed harmonic maps}
\label{sec:flat-conn-pert}

The key to the integrable systems approach to harmonic maps is the
well-known observation that the harmonic map equations amount to the
flatness of a family of connections
\cite{Poh76,Uhl89,ZakMik78,ZahSab79}.  In particular, this gives a
zero-curvature characterisation of Willmore surfaces but more is
true: the constrained Willmore equations also admit a spectral deformation
\cite{BurPedPin02} and hence a zero-curvature representation:
\begin{thm}[\cite{BurCal10}] \label{CWzerocurv} A conformal immersion
$\Lambda$ with central sphere congruence $V$ is constrained Willmore
if and only if there exists a real form
$q\in\Omega^{1}(\Lambda\wedge\Lambda^{(1)})$ such that
\begin{equation}
\label{eq:90}
d^{\lambda}:=\mathcal{D}_{V}+\lambda\mathcal{N}_{V}^{1,0}+\lambda^{-1}\mathcal{N}_{V}^{0,1}+(\lambda^2-1)q^{1,0}+(\lambda^{-2}-1)q^{0,1}
\end{equation}
is flat for all $\lambda\in\C^{\times}$.
\end{thm}

Let us abstract the main properties enjoyed by $V$ and the family of
connections $d^{\lambda}$.  For this, we take
$V\leq\underline{\R}^{n+1,1}$ to be any bundle of $(3,1)$-planes and
consider the corresponding decomposition of the trivial connection:
$d=\cD_V+\cN_V$.  Let $\rho_V$ denote reflection across $V$.

\begin{defn}\label{th:10}
For $k\in\mathbb{N}$, say that $V$ is \emph{$k$-perturbed harmonic}
if there is a family of flat metric connections $d^{\lambda}$,
$\lambda\in\C^{\times}$, on $\underline{\C}^{n+2}$, with the
following properties:
\begin{subequations}\label{eq:1}
\begin{align}
\label{eq:3}
d^1&=d&&\text{(normalisation)}\\
\label{eq:4}
\rho_V\cdot d^{\lambda}&=d^{-\lambda}&&\text{($\rho_V$-twisted)}\\
\label{eq:5}
d^{1/\bar{\lambda}}&=\overline{d^{\lambda}}&&\text{(reality)}
\end{align}
\end{subequations}
and $\lambda\mapsto (d^{\lambda})^{1,0}$ is holomorphic on $\C$ with
a pole of order $k$ at $\infty$ (whence, thanks to \eqref{eq:5},
$\lambda\mapsto(d^{\lambda})^{0,1}$ is holomorphic on
$\C^{\times}\cup\{\infty\}$ with a pole of order $k$ at $0$). In this
case, we say that $(V,d^{\lambda})$ is a $k$-perturbed harmonic
bundle.\footnote{A little circumspection is required here as we have
seen in Remark~\ref{th:35} that a given bundle $V$ can be
$k$-perturbed harmonic in more than one way.}
\end{defn}
In this situation, we may write
\begin{equation*}
d^{\lambda}=\sum_{|i|\leq k}\lambda^iA_i
\end{equation*}
with $A_0$ a real metric connection preserving $V$ and
$A_i\in\Omega^1(\fo(\underline{\C}^{n+2}))$, $i\neq0$, satisfying
$A_i^{1,0}=0$, for $i<0$; $A_{-i}=\overline{A_i}$;
$\mathrm{Ad}_{\rho_V} A_i=(-1)^iA_i$ and
\begin{equation*}
\mathcal{D}_{V}=\sum_iA_{2i}\qquad \mathcal{N}_{V}=\sum_iA_{2i+1}.
\end{equation*}
It follows that $V$ is $1$-perturbed harmonic if and only if it is
harmonic ($A_0=\mathcal{D}_{V}$ and $A_1=\mathcal{N}_{V}^{1,0}$).

\begin{notation}
For $\lambda\in\C\cup\set{\infty}$, set $\lambda^{*}=1/\bar{\lambda}$
so that the reality condition \eqref{eq:5} reads
\begin{equation*}
d^{\lambda^{*}}=\overline{d^{\lambda}}.
\end{equation*}
\end{notation}

With these notions in hand, we reformulate Theorem~\ref{CWzerocurv}
as follows:
\begin{thm}\label{CWzerocurvature}
A conformal immersion $\Lambda$ with central sphere congruence $V$ is
constrained Willmore if and only if $V$ is $2$-perturbed harmonic
with $d^{\lambda}$ satisfying $A_{2}+A_{-2}\in\Omega
^{1}(\Lambda\wedge\Lambda ^{(1)})$.
\end{thm}

\begin{rem}
The condition of $k$-perturbed harmonicity makes sense for maps into
an arbitrary Grassmannian of non-degenerate subspaces of an inner
product space $\R^{p,q}$, or, indeed, an arbitrary pseudo-Riemannian
symmetric space.  Much of what we say below applies in this more
general context.
\end{rem}

\section{Transformations of $k$-perturbed harmonic bundles and
constrained Willmore surfaces}

Our $k$-perturbed harmonic bundles come equipped with
a family of flat connections $d^{\lambda}$.  Such
connections, at least on simply connected open subsets of $\Sigma$,
are gauge equivalent to the trivial connection. These, and other, gauge
transformations can be exploited to produce new $k$-perturbed
harmonic bundles in various ways.  We begin with the simplest of
these transformations.

\subsection{Spectral deformation}
\label{sec:spectral-deformation} Let $(V,d^{\lambda})$ be
$k$-perturbed harmonic.  For $\mu\in S^1$, we have from \eqref{eq:5}
that $d^{\mu}$ is a real, flat, metric connection and so there is, at
least locally, a gauge transformation
$\Phi_{\mu}\in\Gamma(\rO(\underline{\R}^{n+1,1}))$ with
\begin{equation}
\label{eq:24} \Phi_{\mu}\cdot d^{\mu}=d,
\end{equation}
where $\Phi_{\mu}\cdot d^{\mu}=\Phi_{\mu}\circ
d^{\mu}\circ\Phi^{-1}_{\mu}$ is the usual action of gauge
transformations on connections. We note that $\Phi_{\mu}$ is unique
up to left multiplication by a constant element of $\rO(n+1,1)$.

\begin{prop}\label{th:12}
Define $V_{\mu}:=\Phi_{\mu}V$.  Then $V_{\mu}$ is also $k$-perturbed
harmonic with associated flat connections
\begin{equation}\label{eq:25}
d^{\lambda}_{\mu}=\Phi_{\mu}\cdot d^{\lambda\mu}.
\end{equation}
\end{prop}
\begin{proof}
The connections $d^{\lambda}_{\mu}$ are certainly flat, being gauges
of the flat connections $d^{\lambda\mu}$, so we must show that they
have the properties enumerated in Definition~\ref{th:10}.  First,
\eqref{eq:24} tells us that $d^1_{\mu}=d$ so the family is
normalised.  Secondly, we have
$\rho_{V_{\mu}}\circ\Phi_{\mu}=\Phi_{\mu}\circ\rho_V$ which, along
with \eqref{eq:4}, rapidly yields that $d^{\lambda}_{\mu}$ is
$\rho_{V_{\mu}}$-twisted.  Similarly, $\mu=\mu^{*}$ and
$\overline{\Phi_{\mu}}=\Phi_{\mu}$ which, together with \eqref{eq:5},
gives the reality condition for $d^{\lambda}_{\mu}$.  Finally,
$\lambda\mapsto d^{\lambda\mu}$ has the same poles at $0$ and
$\infty$ as $d^{\lambda}$ whence $\lambda\mapsto d^{\lambda}_{\mu}$
does also.
\end{proof}

We can iterate this construction: for $\nu\in S^1$, find a gauge
transformation $\Phi^{\mu}_{\nu}$ for which $\Phi^{\mu}_{\nu}\cdot
d^{\nu}_{\mu}=d$.  Together with \eqref{eq:25}, this gives
\begin{equation*}
(\Phi_{\nu}^{\mu}\Phi_{\mu})\cdot d^{\mu\nu}=d
\end{equation*}
so that we can take $\Phi_{\mu\nu}=\Phi_{\nu}^{\mu}\Phi_{\mu}$ and
conclude:
\begin{prop}
For $\mu,\nu\in S^1$ and $(V,d^{\lambda})$ $k$-perturbed harmonic,
$V_{\mu\nu}=(V_{\mu})_{\nu}$.
\end{prop}
Thus we have an $S^1$-action on $k$-perturbed harmonic bundles up to
congruence.  The analysis applies for $k$-perturbed harmonic maps
into any symmetric space, \emph{mutatis mutandis}, and, for $k=1$,
recovers the well-known result of Terng (c.f. \cite{Uhl89}).

Expanding \eqref{eq:25} in powers of $\lambda$ gives:
\begin{lemma}
In the situation of Proposition~\ref{th:12}, write $d^{\lambda}=\sum
A_i$ and $d^{\lambda}_{\mu}=\sum A^{\mu}_i$.  Then
\begin{subequations}
\begin{align}
A^{\mu}_0&=\Phi_{\mu}\cdot A_0\\
A^{\mu}_i&=\mu^i\Ad_{\Phi_{\mu}}A_i,\label{eq:26}
\end{align}
\end{subequations}
for $i\neq 0$.
\end{lemma}

We use this to examine the special case where $V$ is the central
sphere congruence of a constrained Willmore surface $\Lambda$ with
multiplier $q$.  In this case, set $\Lambda_{\mu}=\Phi_{\mu}\Lambda$
and define $q_{\mu}$ by
\begin{equation*}
q_{\mu}=A^{\mu}_2+A^{\mu}_{-2}=\Ad_{\Phi_{\mu}}(\mu^{2}q^{1,0}+\mu^{-2}q^{0,1}),
\end{equation*}
where the last equality is an instance of \eqref{eq:26}. We now
recover a result of \cite{BurPedPin02} as formulated in
\cite{BurCal10}:
\begin{prop}[\cite{BurPedPin02}]
\label{th:13} %spectral def of cw
For each $\mu\in S^1$, $\Lambda_{\mu}$ is a conformal immersion and
a constrained Willmore surface with multiplier $q_{\mu}$.
\end{prop}
\begin{proof}
To see that $\Lambda_{\mu}$ is a conformal immersion, we need
$\Lambda_{\mu}^{1,0}$ to be maximal isotropic in $V_{\mu}$ . For
$V_{\mu}$ to be the central sphere congruence of $\Lambda_{\mu}$, we
need $\cN_{V_{\mu}}^{1,0}\Lambda_{\mu}^{0,1}=0$ and, finally, we need
$q_{\mu}^{1,0}$ to take values in $\Wedge^2\Lambda_{\mu}^{0,1}$.
Then Theorem \ref{CWzerocurvature} yields the result.  However, since
\eqref{eq:26} gives
$\cN^{1,0}_{V_{\mu}}=\mu\Ad_{\Phi_{\mu}}\cN_V^{1,0}$, all of this
follows as soon as we know that
$\Lambda^{1,0}_{\mu}=\Phi_{\mu}\Lambda^{1,0}$.  For this last,
observe that $q\Lambda=\cN\Lambda=0$ so that the operators $d^{\mu}$
and $d$ coincide on $\Gamma\Lambda$.  We conclude that
$\Lambda^{1,0}=\Lambda\oplus d^{\mu}\sigma(T^{1,0}\Sigma)$, for any
$\sigma\in\Gamma\Lambda$ and, gauging by $\Phi_{\mu}$, we see that
$\Lambda_{\mu}^{1,0}=\Phi_{\mu}\Lambda^{1,0}$ as required.
\end{proof}

\subsection{Dressing action}\label{dressingaction}

We are going to use a version of the Terng--Uhlenbeck dressing
action \cite{TerUhl00} to construct new constrained harmonic bundles
from $V$.  The key idea here is to find $\lambda$-dependent gauge
transformations that preserve the algebraic shape of $d^{\lambda}$
(see \cite{BurCal,BurDonPedPin11} for similar viewpoints on
dressing).  These transformations, when applied to the central
sphere congruence of a constrained Willmore surface, will give rise
to new such surfaces.

We begin by specifying the properties of this gauge transformation.
\begin{defn}\label{th:9}
Let $(V,d^{\lambda})$ be a $k$-perturbed harmonic bundle.  A
\emph{dressing gauge for $(V,d^{\lambda})$} is a family of gauge
transformations $\lambda\mapsto r(\lambda)\in\Gamma\rO(\unC^{n+2})$
which is holomorphic in $\lambda$ near $0,\infty\in\P^{1}$ and has
the following properties:
\begin{enumerate}
\item \label{item:1} $r(-\lambda)\circ\rho_V\circ r(\lambda)^{-1}$ is
independent of $\lambda\in\dom(r)$.
\item \label{item:2} For all $\lambda\in\dom(r)$,
\begin{equation}
\label{eq:8}
r(\lambda^{*})=\overline{r(\lambda)}\qquad\text{(reality)}.
\end{equation}
(Thus we require that $\dom(r)$ is stable under
$\lambda\mapsto\pm\lambda^{*}$.)
\item \label{item:3} The connections $r(\lambda)\cdot d^{\lambda}$
extend from $\dom(r)\setminus\set{0,\infty}$ to a holomorphic family
of connections $\hd^{\lambda}$ on $\C^{\times}$ with $\hd^1=d$.
\item \label{item:4}
\begin{equation}
\label{eq:31} \det(r(0)^{-1}r(\infty)_{|V})=1.
\end{equation}
\end{enumerate}
For such an $r$, set $\hV=r(0)V$ and call $(\hV,\hd^{\lambda})$ the
\emph{dressing transform of $(V,d^{\lambda})$ by $r$}.
\end{defn}
Note that item (\ref{item:1}) now reads
\begin{equation}
\label{eq:7} r(-\lambda)\circ\rho_V=\rho_{\hV}\circ
r(\lambda),\qquad\text{(twisted)}
\end{equation}
for $\lambda\in\dom(r)$, and then evaluating at $\lambda=\infty$
yields $\hV=r(\infty)V$ also.  In particular,
$r(0)^{-1}r(\infty)V=V$ so that \eqref{eq:31} makes sense.  

It is our contention that $(\hV,\hd^{\lambda})$ is again a
$k$-perturbed harmonic map $\Sigma\to\mathrm{Gr}_{3,1}(\R^{n+1,1})$.
We begin by showing that $\hV$ is indeed (the complexification of) a
bundle of real $(3,1)$-planes.  First, evaluate \eqref{eq:8} at
$\lambda=0$, to get $r(\infty)=\overline{r(0)}$ so that
$\overline{\hV}=r(\infty)V=\hV$.  We now have:
\begin{lemma}\label{th:22}
Let $W\leq\C^{n+2}$ be a real (thus $W=\overline{W}$), non-degenerate
$4$-plane and $T\in\rO(\C^{n+2})$.  Then
\begin{enumerate}
\item \label{item:9}$W\cap\R^{n+1,1}$ is a $(3,1)$-plane if and only if $\dim
(U\cap\overline{U})=1$, for any maximal isotropic $2$-plane $U\leq W$.
\item \label{item:8}If $W\cap\R^{n+1,1}$ is a $(3,1)$-plane and
$TW=\overline{TW}$ then $TW\cap\R^{n+1,1}$ is also a $(3,1)$-plane if and
only if $\det(T^{-1}\overline{T})=1$.
\end{enumerate}
\end{lemma}
\begin{proof}
For (\ref{item:9}), note that $W_{\R}:=W\cap\R^{n+1,1}$ has signature
$(4,0)$ or $(3,1)$.  Since $U\cap\overline{U}\cap\R^{n+1,1}$ is a
real isotropic subspace of $W_{\R}$, we have
$\dim(U\cap\overline{U})\leq 1$ with equality forcing $W_{\R}$ to
have signature $(3,1)$.  For the converse, if
$\dim(U\cap\overline{U})=0$, then $U$ and $\overline{U}$ are the
$\pm\sqrt{-1}$-eigenspaces of an orthogonal complex structure on
$W_{\R}$ and this requires that $W_{\R}$ be a $(4,0)$-plane.

For (\ref{item:8}), let $U\leq W$ be maximal isotropic so that $TU$
is maximal isotropic in $TW$.  We have just seen that $(TW)_{\R}$ is a
$(3,1)$-plane if and only if $\dim(TU\cap \overline{TU})=1$, or,
equivalently, $\dim U\cap (T^{-1}\overline{T})\overline{U}=1$.  Now
$U$ and $\overline{U}$ define lines in the quadric in
$\P(W)$ given by the ambient inner product and such lines intersect
exactly when they lie in different rulings of that quadric.  These
rulings comprise different $\rSO(W)$-orbits which are permuted by
$\rO(W)\setminus\rSO(W)$.  Thus $\det(T^{-1}\overline{T})=1$ if and
only if $\overline{U}$ and $(T^{-1}\overline{T})\overline{U}$ lie in
the same ruling if and only if $(T^{-1}\overline{T})\overline{U}$
intersects $U$ in a single (projective) point.
\end{proof}

Applying this last fibrewise to $V$ with $T=r(0)$ and using
\eqref{eq:31} immediately yields:
\begin{corol}
\label{th:4} Let $(\hV,\hd^{\lambda})$ be the dressing transform of
$(V,d^{\lambda})$ by $r$.  Then,
\begin{enumerate}
\item \label{item:5} $\overline{\hV}=\hV$.
\item \label{item:6} $\hV\cap\underline{\R}^{n+1,1}$ is a bundle of $(3,1)$-planes.
\item \label{item:7} For any maximal isotropic subbundle of $U$ of $\hV$,
$U\cap\overline{U}$ has rank $1$.
\end{enumerate}
\end{corol}

With these preliminaries dealt with, we have:
\begin{thm}\label{th:1}
Let $(\hV,\hd^{\lambda})$ be the dressing transform of
$(V,d^{\lambda})$ by $r$.  Then $(\hV,\hd^{\lambda})$ is
$k$-perturbed harmonic.
\end{thm}
\begin{proof}
This amounts to showing that the connections $\hat{d}^{\lambda}$
satisfy the conditions of Definition~\ref{th:10}.  By hypothesis,
$\hat{d}^{\lambda}$ is normalised.  For the remaining conditions that
are pointwise in $\lambda$, we establish them first for
$\lambda\in\dom(r)\setminus\set{0,\infty}$ and then conclude they
hold for all $\lambda\in\C^{\times}$ by unique continuation. Indeed,
\eqref{eq:7} together with the fact that $d^{\lambda}$ is
$\rho_V$-twisted combine to show that $\hat{d}^{\lambda}$ is
$\rho_{\hat{V}}$-twisted on $\dom(r)\setminus\set{0,\infty}$.
Similarly, the reality condition on $\hat{d}^{\lambda}$ follows
immediately from that on $d^{\lambda}$ and \eqref{eq:8}. Moreover,
for $\lambda\in\dom(r)\setminus\set{0,\infty}$, $\hat{d}^{\lambda}$
is flat since it is a gauge of the flat connection $d^{\lambda}$.

Finally, the pole behaviour of $\hat{d}^{\lambda}$ at zero and
infinity coincides with that of $d^{\lambda}$ since the connections
differ by the gauge transformations $r(\lambda)$ which are
holomorphic in $\lambda$ near those points.
\end{proof}

To go further, we need to establish some relations between the
coefficients $\hat{A}_i$ of $\hat{d}^{\lambda}$ and those of
$d^{\lambda}$.  First let us establish some notation: write 
\begin{align*}
A_{+}&=\sum_{0<i\leq k}\lambda^iA_i,& A_{-}&=\sum_{k\leq i<0}\lambda^iA_i
\end{align*}
so that $d^{\lambda}=A_{+}+A_0+A_{-}$ and, similarly,
$\hat{d}^{\lambda}=\hat{A}_{+}+\hat{A}_0+\hat{A}_{-}$.  Moreover,
define $\chi_0,\chi_{\infty}$ by
\begin{align*}
\chi_0&=r^{-1}\del r/\del\lambda,\\
\chi_\infty&=r^{-1}\del r/\del\mu,
\end{align*}
where $\mu=1/\lambda$. We note that
$\chi_{\infty}(\lambda)=\overline{\chi_0(\lambda^{*})}$, for
$\lambda\in\C^{\times}$.  On $\dom(r)\setminus\set{0,\infty}$, we
have
\begin{equation}
\label{eq:86}
\hd^{\lambda}=r(\lambda)\cdot d^{\lambda}.
\end{equation}
The $(1,0)$-part of this is holomorphic at $\lambda=0$ and
differentiating with respect to $\lambda$ yields
\begin{equation}
\label{eq:87}
\del \hA_{+}/\del\lambda=
\Ad_{r}(\del A_{+}/\del\lambda-A_0^{1,0}\chi_0-[A_{+},\chi_{0}]).
\end{equation}
Again, we may write the $(0,1)$-part of \eqref{eq:86} as
\begin{equation}
\label{eq:88}
\hat{A}_0^{0,1}+\hat{A}_{-}=r\cdot A^{0,1}_0+\Ad_rA_{-},
\end{equation}
from which we conclude that $\hat{A}_{-}-\Ad_rA_{-}$ is holomorphic
near $\lambda=0$.

Taking coefficients of powers of $\lambda$ now yields:
\begin{lemma}\label{th:2}
In the situation of Theorem~\ref{th:1}, with
$d^{\lambda}=\sum_{|i|\leq k}\lambda^iA_i$ and
$\hd^{\lambda}=\sum_{|i|\leq k}\lambda^i\hA_i$, we have:
\begin{subequations}\label{eq:15}
\begin{align}
\label{eq:10} \hA^{1,0}_0&=r(0)\cdot A^{1,0}_0&
\hA^{0,1}_0&=r(\infty)\cdot A^{0,1}_0\\
\label{eq:2} \hA_1&=\Ad_{r(0)}(A_1-A_0^{1,0}\chi_0(0))&
\hA_{-1}&=\Ad_{r(\infty)}(A_{-1}-A_0^{0,1}\chi_{\infty}(\infty))\\
\label{eq:6} \hA_{-k}&=\Ad_{r(0)}A_{-k}&
\hA_k&=\Ad_{r(\infty)}A_k\\
\intertext{and, for $k>1$,} \label{eq:9}
\hA_{-k+1}&=\Ad_{r(0)}(A_{-k+1}+[\chi_0(0),A_{-k}])&
\hA_{k-1}&=\Ad_{r(\infty)}(A_{k-1}+[\chi_{\infty}(\infty),A_k]).
\end{align}
Finally,
\begin{align}\label{eq:12}
\hA_{0}^{1,0}&=r(\infty)\cdot A_{0}^{1,0}+
\frac{1}{k!}\frac{\del^{k}}{\del\mu^k}_{|\mu=0}\Ad_{r}(A_k+\dots+\mu^{k-1}
A_{1})\\\notag 
\hA_{0}^{0,1}&=r(0)\cdot A_{0}^{0,1}+
\frac{1}{k!}\frac{\del^{k}}{\del\lambda^{k}}_{|\lambda=0}\Ad_{r}(A_{-k}+\dots+\lambda^{k-1}
A_{-1}).
\end{align}
\end{subequations}
\end{lemma}
\begin{proof}
Evaluate the $(1,0)$ part of \eqref{eq:86} at $\lambda=0$ to get 
\eqref{eq:10}$^{1,0}$ and evaluate \eqref{eq:87} at $\lambda=0$ to
get \eqref{eq:2}$^{1,0}$.

Meanwhile, comparing coefficients of $\lambda $ in \eqref{eq:88} yields
\begin{align*}
\hA_{-k+i}&=\frac{1}{i!}\frac{\del^{i}}{\del\lambda^i}_{|\lambda=0}\Ad_{r(\lambda)}(A_{-k}+\dots+\lambda^{k-1}
A_{-1}),\text{for $0\leq i<k$}\\
\hA_0^{0,1}&=r(0)\cdot A_{0}^{0,1}+
\frac{1}{k!}\frac{\del^{k}}{\del\lambda^{k}}_{|\lambda=0}\Ad_{r(\lambda)}(A_{-k}+\dots+\lambda^{k-1}
A_{-1}).
\end{align*}
The last equation is \eqref{eq:12}$^{0,1}$ while the cases $i=0,1$
of the first yield \eqref{eq:6}$^{0,1}$ and \eqref{eq:9}$^{0,1}$.  A
similar argument at $\mu=0$ (or an appeal to the reality of
$\hd^{\lambda}$) gives the remaining equations.
\end{proof}

We pause for a short diversion of possibly independent interest: our
dressing transformation changes a certain energy density by an
\emph{exact} $2$-form.  First a definition:
\begin{defn}\label{th:33}
Let $(V,d^{\lambda})$ be $k$-perturbed harmonic.  The \emph{energy
density} $e(V)$ of $(V,d^{\lambda})$ is the $2$-form on $\Sigma$ given by
\begin{equation*}
e(V)=i\sum_{0<j\leq k}j(A_j\wedge A_{-j}).
\end{equation*}
\end{defn}
Note that, when $V$ is harmonic (thus $k=1$), $e(V)$ is the usual
harmonic map energy density:
\begin{equation*}
e(V)=\tfrac{1}{2}(\cN_{V}\circ J^{\Sigma}\wedge\cN_V).
\end{equation*}

We now have:
\begin{prop}\label{th:32}
\label{th:3} Let $(V,d^{\lambda})$ and $(\hV,\hd^{\lambda})$ be
$k$-perturbed harmonic bundles with $\hV$ the dressing transform of
$V$ by $r$. Then
\begin{equation*}
e(\hV)=e(V)-id\Res_{\lambda=0}(\chi_0,A_{-}).
\end{equation*}
\end{prop}
\begin{proof}
We begin by observing that our energy density is a residue:
\begin{equation*}
\sum_{0<j\leq k}j(A_j\wedge A_{-j})=\Res_{\lambda=0}(\del
A^{+}/\del\lambda\wedge A^{-}).
\end{equation*}
Now $\hat{A}^{-}-\Ad_rA^{-}$ is holomorphic near $\lambda=0$ so that
\eqref{eq:87} yields:
\begin{align*}
\Res_{\lambda=0}(\del\hA^{+}/\del\lambda\wedge\hA^{-})&=
\Res_{\lambda=0}\bigl(\Ad_r(\del
A^{+}/\del\lambda-A_0^{1,0}\chi_{0}-[A_{+},\chi_{0}])\wedge
\Ad_{r}A^{-}\bigr)\\
&=\Res_{\lambda=0}\bigl(\del
A^{+}/\del\lambda-A_0^{1,0}\chi_{0}-[A_{+},\chi_{0}]\wedge
A^{-}\bigr).
\end{align*}
Since $A_{-}$ is a $(0,1)$-form, $(A_0^{1,0}\wedge A_{-})=(A_0\wedge
A_{-})$ and we have
\begin{align*}
(A_0^{1,0}\chi_{0}+[A_{+},\chi_{0}]\wedge A^{-})&=
(A_0\chi_{0}+[A_{+},\chi_{0}]\wedge
A^{-})\\
&=d(\chi_0,A_{-})-(\chi_0,d^{A_0}A_{-})-(\chi_0,[A_{+}\wedge
A_{-}])\\
&=d(\chi_0,A_{-})+(\chi_0,R^{A_0}+d^{A_0}A_{+}),
\end{align*}
where, for the last equality, we have used the flatness of
$d^{\lambda}$.  Since $(\chi_0,R^{A_0}+d^{A_0}A_{+})$ is holomorphic
near $\lambda=0$ and so has no residue there, we conclude:
\begin{equation*}
\Res_{\lambda=0}(\del\hA^{+}/\del\lambda\wedge\hA^{-})=
\Res_{\lambda=0}\bigl(\del
A^{+}/\del\lambda\wedge A_{-})-d\Res_{\lambda=0}(\chi_0,A_{-})
\end{equation*}
and the result follows at once.
\end{proof}
\begin{corol}
Let $V,\hV$ be harmonic bundles with $\hV$ the dressing transform of
$V$.  Then the harmonic map energy densities of $V$ and $\hV$
differ by an exact $1$-form.
\end{corol}

So far, just as in section~\ref{sec:spectral-deformation}, our
analysis could be applied \emph{mutatis mutandis} to $k$-perturbed
harmonic maps into any symmetric space but now we specialise to the
case of particular interest to us: we suppose that $V$ is the
central sphere congruence of a constrained Willmore surface
$\Lambda$ with multiplier $q$.  In this case, Theorem~\ref{th:1}
yields a $2$-perturbed harmonic $\hV$ and we are going to show that
$\hV$ is also the central sphere congruence of a constrained
Willmore surface $\hL$.

Our candidates for a new constrained Willmore surface and
corresponding multiplier are
\begin{subequations}
\label{eq:32}
\begin{align}\label{eq:72}
\hL&=r(0)\Lambda^{1,0}\cap r(\infty)\Lambda^{0,1},\\
\hq&=\hA_{2}+\hA_{-2}=\Ad_{r(\infty)}q^{1,0}+\Ad_{r(0)}q^{0,1},\label{eq:73}
\end{align}
\end{subequations}
where the last identity is \eqref{eq:6}.  Since
$\overline{r(0)\Lambda^{1,0}}=r(\infty)\Lambda^{0,1}$,
Lemma~\ref{th:4}(\ref{item:7}) assures us that $\hL$ is a real line
subbundle of $\unC^{n+2}$.

We now have the main result of this section:
\begin{thm}
\label{th:5} Let $(V,d^{\lambda})$ be the $2$-perturbed harmonic
central sphere congruence of a constrained Willmore surface
$\Lambda$ with multiplier $q$.  Let $(\hV,\hd^{\lambda})$ be the
dressing transform of $(V,d^{\lambda})$ by $r$ and define $\hL$,
$\hq$ by \eqref{eq:32}.

Then $\hL$ is conformal and a constrained Willmore surface with
multiplier $\hq$ and central sphere congruence $\hV$ on the open set
where it immerses.
\end{thm}
We therefore extend the terminology of Definition~\ref{th:9} and say
that $\hL$ is \emph{the dressing transform of $\Lambda$ by $r$}.
\begin{proof}
There are three things to prove here.  For conformality, we must
show that $\hL^{1,0}$ is isotropic.  To see that $\hL$ is
constrained Willmore with multiplier $\hq$, we need to show that
$\hq^{1,0}\in\Omega^{1,0}(\Wedge^2\hL^{0,1})$ so that $\hq$ gives
rise to a holomorphic quadratic differential and that
$\cN_{\hV}^{1,0}\hL^{0,1}=0$ so that $\hV$ is the central sphere
congruence of $\hL$.  Theorem~\ref{CWzerocurvature}
then establishes the conclusion.

For all this, we prove that $\hL^{0,1}=r(\infty)\Lambda^{0,1}$, whence, by
reality, $\hL^{1,0}=r(0)\Lambda^{1,0}$, and
$\cN_{\hV}^{0,1}r(0)\Lambda^{1,0}=0$.  It is convenient to prove
the second of these assertions first.  From \eqref{eq:9}, we have
\begin{equation}
\cN_{\hV}^{0,1}=\Ad_{r(0)}(\cN_V^{0,1}+[\chi_0(0),q^{0,1}])
\end{equation}
so that we need to establish that
\begin{equation*}\label{eq:16}
(\cN_V^{0,1}+[\chi_0(0),q^{0,1}])\Lambda^{1.0}=0.
\end{equation*}
However, $\cN_V^{0,1}\Lambda^{1,0}$ vanishes, since $V$ is the
central sphere congruence of $\Lambda$, leaving us with  the
$[\chi_0(0),q^{0,1}]$ term.  However, differentiating \eqref{eq:7} at
$\lambda=0$ shows that $\chi_0(0)$ anti-commutes with $\rho_V$ and so
takes values in $V\wedge V^{\perp}$.  On the other hand, $q^{0,1}$
takes values in $\Wedge^2\Lambda^{1,0}$ so that their bracket lies
in $\Lambda^{1,0}\wedge V^{\perp}$ and so annihilates
$\Lambda^{1,0}$ as required.  Similarly,
$\cN_{\hV}^{1,0}r(\infty)\Lambda^{0,1}=0$ and, in particular,
$\cN_{\hV}\hL$ vanishes.

Again $\hq^{0,1}r(0)\Lambda^{1,0}=r(0)q^{0,1}\Lambda^{1,0}=0$ and,
similarly, $\hq^{1,0}r(\infty)\Lambda^{0,1}$ vanishes.  In
particular, $\hq\hL=0$.

It follows that the operators $d^{0,1}$, $\cD_{\hV}^{0,1}$ and
$(\cD_{\hV}-\hq)^{0,1}$ all coincide on $\Gamma\hL$.  However,
\eqref{eq:10} reads
\begin{equation*}
(\cD_{\hV}-\hq)^{0,1}=r(\infty)\cdot(\cD_V-q)^{0,1}
\end{equation*}
and $\Lambda^{0,1}$ is $(\cD_V-q)^{0,1}$-stable whence
$r(\infty)\Lambda^{0,1}$ is $(\cD_{\hV}-\hq)^{0,1}$-stable.  It
follows at once that
$d^{0,1}:\Gamma\hL\to\Omega^{0,1}(r(\infty)\Lambda^{0,1})$ so that
$\hL^{0,1}=r(\infty)\Lambda^{0,1}$ and, similarly,
$\hL^{1,0}=r(0)\Lambda^{1,0}$, since $\hL$ immerses.
\end{proof}
We remark that $q$ vanishes exactly when $\hq$ does so that our
dressing transforms when applied to Willmore surfaces give Willmore
surfaces once more.

Proposition~\ref{th:32} applies in the current setting:
\begin{corol}
Let $\hat{\Lambda}$ be a dressing transform of a constrained Willmore
surface $\Lambda$.  Then the Willmore densities of $\Lambda$ and
$\hat{\Lambda}$ differ by an exact $1$-form.
\end{corol}
\begin{proof}
We know from section~\ref{sec:constr-willm-surf} that the Willmore density of a surface
coincides with the harmonic map energy
$\frac{1}{2}(\cN_V\circ J^{\Sigma}\wedge\cN_V)$ of its central sphere congruence
$V$.  On the other hand, in the present situation, the energy $e(V)$
of definition~\ref{th:33} is given by
\begin{equation*}
e(V)=\frac{1}{2}(\cN_V\circ J\Sigma\wedge\cN_V)+(q\circ J^{\Sigma}\wedge q)=
\frac{1}{2}(\cN_V\circ J^{\Sigma}\wedge\cN_V),
\end{equation*}
since $q$ takes values in the isotropic bundle
$\Lambda\wedge\Lambda^{\perp}$.
\end{proof}

We conclude this analysis by showing that the holomorphic quadratic
differentials that are the Lagrange multipliers for $\Lambda$ and
$\hL$ coincide. First, we extend the scope of \eqref{eq:Q}:
\begin{lemma}
\label{th:36}
Let $\Lambda$ be a conformal immersion with central sphere congruence
$V$.  Let $q\in\Omega^1(\Lambda\wedge\Lambda^{(1)}$ corresponding to
a holomorphic quadratic differential $Q$ via \eqref{eq:Q}.  Then, for
any $\tau\in\Gamma\Lambda^{1,0}$, we have
\begin{equation}
\label{eq:91}
Q^{0,2}\tau=q^{0,1}(\cD_V^{0,1}\tau).
\end{equation}
\end{lemma}
\begin{proof}
Choose $Z\in T^{1,0}\Sigma$ and $\sigma\in\Gamma\Lambda$ with
$(d_Z\sigma,d_{\bar{Z}}\sigma)=1$.  Since
$\Lambda\wedge\Lambda^{1,0}=\bigwedge^2\Lambda^{1,0}$ has rank $1$,
$q_{\bar{Z}}$ is a multiple of $\sigma\wedge d_{Z}\sigma$ and
\eqref{eq:Q} quickly yields
\begin{equation}
\label{eq:92}
q_{\bar{Z}}=-Q(\bar{Z},\bar{Z})\sigma\wedge d_Z\sigma.
\end{equation}
The right hand side of \eqref{eq:91} is tensorial in $\tau$ so that,
in view of \eqref{eq:Q}, it suffices to check \eqref{eq:91} with
$\tau=d_{Z}\sigma$.  But
\begin{align*}
(\sigma\wedge d_Z\sigma)(\cD_{V})_{\bar{Z}}d_Z\sigma&=
(\sigma,(\cD_{V})_{\bar{Z}}d_Z\sigma)d_Z\sigma-
(d_Z\sigma,(\cD_{V})_{\bar{Z}}d_Z\sigma)\sigma\\
&=-(d_{\bar{Z}}\sigma,d_Z\sigma)d_Z\sigma+\half
d_{\bar{Z}}(d_Z\sigma,d_Z\sigma)\\
&=-d_Z\sigma
\end{align*}
and the result follows at once from \eqref{eq:92}.
\end{proof}
\begin{prop}
\label{th:6} Let $\Lambda$ be a constrained Willmore surface and
$\hL$ the dressing transform of $\Lambda$ by $r$.  Then
$Q_{\hL}=Q_{\Lambda}$.
\end{prop}
\begin{proof}
From Lemma~\ref{th:36}, for any section $\hat{\tau}$ of $\hL^{1,0}$,
we have
\begin{equation*}
Q^{0,2}_{\hL}\hat{\tau}=\hq^{0,1}(\cD_{\hV}^{0,1}\hat{\tau})
=\hq^{0,1}((\cD_{\hV}-\hq)^{0,1}\hat{\tau}).
\end{equation*}
We write $\hat{\tau}=\Ad_{r(0)}\tau$, for
$\tau\in\Gamma\Lambda^{1,0}$, and apply \eqref{eq:15} to see that
\begin{align*}
\hq^{0,1}((\cD_{\hV}-\hq)^{0,1}\hat{\tau})&=r(0)q^{0,1}((\cD_V-q)^{0,1}\tau)+
\hq^{0,1}(\frac{1}{2}\frac{\del^2}{\del\lambda^2}_{|\lambda=0}
\Ad_{r(\lambda)}(q^{0,1}+\lambda\cN_V^{0,1})\hat{\tau})\\
&=Q^{0,2}_{\Lambda}\hat{\tau}+
\hq^{0,1}(\frac{1}{2}\frac{\del^2}{\del\lambda^2}_{|\lambda=0}
\Ad_{r(\lambda)}(q^{0,1}+\lambda\cN_V^{0,1})\hat{\tau})
\end{align*}
so that the issue is to show that the last term vanishes.

For this, note that
\begin{equation}\label{eq:48}
\frac{\del^2}{\del\lambda^2}_{|\lambda=0}
\Ad_{r(\lambda)}(q^{0,1}+\lambda\cN_V^{0,1})=
\Ad_{r(0)}((\del/\del\lambda+\ad
\chi_0)^{2}(q^{0,1}+\lambda\cN_V^{0,1})_{\lambda=0})
\end{equation}
so that it suffices to show that $C:=(\del/\del\lambda+\ad
\chi_0)^{2}(q^{0,1}+\lambda\cN_V^{0,1})_{\lambda=0}$ preserves
$\Lambda^{1,0}$.  We may write \eqref{eq:48} as
\begin{equation*}
\Ad_{r(0)}C=\frac{\del^2}{\del\lambda^2}_{|\lambda=0}g(\lambda),
\end{equation*}
where $g(\lambda)=\Ad_{r(\lambda)}(q^{0,1}+\lambda\cN^{0,1}_V)$.  From
\eqref{eq:7}, we have $\Ad_{\rho_{\hV}}g(\lambda)=g(-\lambda)$ and,
differentiating this twice, we learn that $C$ commutes with $\rho_V$
and so preserves $V$.  Since $\Lambda^{1,0}$ is maximal isotropic in
$V$, we are reduced to showing that $(C\tau_1,\tau_2)$ vanishes for
all $\tau_1,\tau_2\in\Lambda^{1,0}$.  Now set
$B_{\lambda}=(\del/\del\lambda+\ad \chi_0)(q^{0,1}+\lambda\cN_V^{0,1})$
so that $C=(\del/\del\lambda+\ad \chi_0)_{|\lambda=0}B_{\lambda}$.  We
observe that
\begin{equation*}
B_0=\cN_V^{0,1}+[\chi_0(0),q^{0,1}]
\end{equation*}
which, from \eqref{eq:16}, annihilates $\Lambda^{1,0}$.  We compute:
\begin{align*}
(C\tau_1,\tau_2)&=(\del/\del\lambda_{|\lambda=0}B_{\lambda}\tau_1,\tau_2)+([\chi_0(0),B_0]\tau_{1},\tau_2)\\
&=(\del/\del\lambda_{|\lambda=0}B_{\lambda}\tau_1,\tau_2)-
(B_0\tau_1,\chi_0(0)\tau_2)+(\chi_0(0)\tau_1,B_0\tau_{2})\\
&=\del/\del\lambda_{|\lambda=0}(B_{\lambda}\tau_1,\tau_2).
\end{align*}
However, $B_{\lambda}=(\del/\del\lambda+\ad \chi_0)A_{\lambda}$ for
$A_{\lambda}=q^{0,1}+\lambda\cN_V^{0,1}$ which last annihilates
$\Lambda^{1,0}$, for all $\lambda$, so that repeating the last
computation with $B_{\lambda}$ gives
\begin{equation*}
(B_{\lambda}\tau_1,\tau_2)=\del/\del\lambda(A_{\lambda}\tau_1,\tau_2)
\end{equation*}
which vanishes identically.
\end{proof}

\subsection{B\"{a}cklund transformation}
\label{sec:backl-transf}

Let us now remedy a lack in the analysis of the last section by providing
examples of dressing gauges.  For this, we follow Terng--Uhlenbeck
\cite{TerUhl00} and contemplate dressing by \emph{simple factors}.

Here is the basic building block of the construction: for $L^{+},
L^{-}$ null line subbundles of $\unC^{n+2}$ which are
\emph{complementary} in the sense that $L^{+}_x$ and $L^{-}_x$ are
not orthogonal, for all $x\in\Sigma$, define
$\Gamma^{L^{+}}_{L^{-}}:\C^{\times}\to\Gamma(\rO(\unC^{n+2}))$ by
\begin{equation*}
\Gamma^{L^{+}}_{L^{-}}(\lambda)=
\begin{cases}
\lambda&\text{on $L^{+}$;}\\
1&\text{on $(L^{+}\oplus L^{-})^{\perp}$;}\\
\lambda^{-1}&\text{on $L^{-}$.}
\end{cases}
\end{equation*}
The key property enjoyed by these gauge transformations is that
$\Ad\Gamma^{L^{+}}_{L^{-}}(\lambda)$ is semisimple with eigenvalues
$\lambda$, $1$ and $\lambda^{-1}$ only.

Our simple factors will be constructed by precomposing these gauge
transformations with a linear fractional transformation of $\lambda$
so the following simple lemma will be important for us:
\begin{lemma}
\label{th:7} Let $\lambda\mapsto d^{\lambda}$,
$\lambda\in\C^{\times}$, be a family of metric connections on
$\unC^{n+2}$, holomorphic in $\lambda$ on an open subset of $\P^1$.
Let $\alpha,\beta\in\P^1$ and $\psi_{\beta}^{\alpha}:\P^1\to\P^1$ a
linear fractional transformation with a zero at $\alpha$ and a pole
at $\beta$.  If $d^{\lambda}$ is holomorphic at $\alpha$, then the
gauged family of connections
\begin{equation*}
\Gamma^{L^{+}}_{L^{-}}(\psi^{\alpha}_{\beta}(\lambda))\cdot
d^{\lambda}
\end{equation*}
extend holomorphically across the singularity at $\alpha$ if and
only if $L^{+}$ is $d^{\alpha}$-parallel.  Similarly, if
$d^{\lambda}$ is holomorphic at $\beta$, the connections extend
holomorphically at $\beta$ if and only if $L^{-}$ is
$d^{\beta}$-parallel.
\end{lemma}
\begin{proof}
Write $W=(L^{+}\oplus L^{-})^{\perp}$ so that we have a
decomposition $\unC^{n+2}=L^{-}\oplus W\oplus L^{+}$.  We have a
corresponding decomposition of $d^{\alpha}$:
\begin{equation*}
d^{\alpha}=D+\beta^{+}+\beta^{-}
\end{equation*}
where $D$ preserves $L^{\pm}$ and
$\beta^{\pm}\in\Omega^1(L^{\pm}\wedge W)$.  We note that
$L^{\pm}\wedge W$ are the $\lambda^{\pm 1}$-eigenbundles of $\Ad
\Gamma^{L^{+}}_{L^{-}}(\lambda)$.

Now write $d^{\lambda}=d^{\alpha}+(\lambda-\alpha)B(\lambda)$ (we
assume that $\alpha\neq\infty$, otherwise we work with $1/\lambda$).
Then, with $\Gamma(\lambda)=\Gamma^{L^{+}}_{L^{-}}(\psi(\lambda))$
\begin{equation*}
\Gamma(\lambda)\cdot d^{\lambda}= \Gamma(\lambda)\cdot d^{\alpha}+
(\lambda-\alpha)\Ad_{\Gamma(\lambda))}B(\lambda).
\end{equation*}
The second term is holomorphic near $\alpha$ since
$\Ad_{\Gamma(\lambda)}$ introduces at most a simple pole at
$\alpha$.  As for the first term,
\begin{align*}
\Gamma(\lambda)\cdot d^{\alpha}&=\Gamma(\lambda)\cdot
D+\Ad_{\Gamma(\lambda)}(\beta^{+}+\beta^{-})\\
&=D+\psi(\lambda)\beta^{+}+\psi(\lambda)^{-1}\beta^{-},
\end{align*}
which is holomorphic near $\alpha$ if and only if $\beta^{-}$
vanishes.  However, this is the case if and only if $L^{+}$ is
$d^{\alpha}$-parallel.

The case at $\beta$ now follows after noting that
$\Gamma^{L^{+}}_{L^{-}}(\psi^{\alpha}_{\beta}(\lambda))=
\Gamma^{L^{-}}_{L^{+}}(1/\psi^{\alpha}_{\beta}(\lambda))$ to which
the first case applies.
\end{proof}

Simple factors have only two poles so, in order to take care of both
the twisting and reality conditions\footnote{These conditions mean
that if $\alpha$ is a pole, so is $-\alpha$ and
$\pm1/\bar{\alpha}$.} that Definition~\ref{th:9} imposes, we shall
need a product of two simple factors
(c.f.~\cite{Pac10,KobIno05,Mah02}) and have recourse to a Bianchi
permutability result to account for the non-commutativity of that
product.  For this and related results, the following lemma  will be
useful.
\begin{lemma}
\label{th:11}%products of simple factors
Let $\ell^{+},\ell^{-}$ and $\hat{\ell}^{+},\hat{\ell}^{-}$ be two
pairs of complementary null lines in $\C^{n+2}$. Let
$\psi^{\alpha}_{\beta}:\P^1\to\P^1$ be a linear fractional
transformation with zero at $\alpha$ and pole at $\beta$.  Finally,
let $E:\lambda\mapsto E(\lambda)\in\rO(\C^{n+2})$ be holomorphic
near $\alpha$ and $\beta$.  Then
\begin{equation}\label{eq:13}
\lambda\mapsto
\Gamma^{\hat{\ell}^{+}}_{\hat{\ell}_{-}}(\psi^{\alpha}_{\beta}(\lambda))E(\lambda)
\bigl(\Gamma^{\ell^{+}}_{\ell_{-}}(\psi^{\alpha}_{\beta}(\lambda))\bigr)^{-1}
\end{equation}
is holomorphic at $\alpha$ if and only if
$E(\alpha)\ell^{+}=\hat{\ell}^{+}$ and holomorphic at $\beta$ if and
only if $E(\beta)\ell^{-}=\hat{\ell}^{-}$.
\end{lemma}
\begin{proof}
Consider first the case where
$\psi^{\alpha}_{\beta}(\lambda)=\lambda$.  Here, holomorphicity of
\eqref{eq:13} at $0$ is precisely the statement of
\cite[Lemma~4.10]{Bur06} while holomorphicity at $\infty$ follows by
swapping the roles of $\ell^{\pm}$ and replacing $\lambda$ with
$1/\lambda$.

The case of arbitrary $\alpha,\beta$ can now be reduced to the first
case by precomposing \eqref{eq:13} with
$(\psi^{\alpha}_{\beta})^{-1}$.
\end{proof}
\begin{rem}
Lemma~\ref{th:11} can be viewed as a discrete analogue of
Lemma~\ref{th:7} and, indeed, for flat $d^{\lambda}$, implies
Lemma~\ref{th:7} by letting $E(\lambda)$ be a gauge transform
relating $d^{\lambda}$ and $d$.
\end{rem}

With all this in hand, let $V$ be a nondegenerate subbundle of
$\unC^{n+2}$, $\alpha\in\C^{\times}$ and $L$ a null, line subbundle
such that $L$ and $\rho_VL$ are complementary on an open set.  From
this data, define gauge transformations $p^{V}_{\alpha,L}(\lambda)$
on that open set by
\begin{equation*}
p^{V}_{\alpha,L}(\lambda)=\Gamma^{L}_{\rho_VL}\bigl(
\tfrac{(1+\alpha)(\lambda-\alpha)}{(1-\alpha)(\lambda+\alpha)}\bigr).
\end{equation*}
We observe:
\begin{subequations}
\label{eq:17}
\begin{gather}
\label{eq:20}
p^{V}_{\alpha,L}(1)=1,\\
\label{eq:18} p^V_{\alpha,L}(-\lambda)\circ\rho_V\circ
p^{V}_{\alpha,L}(\lambda)^{-1
}=
\rho_{V'},\\
\intertext{where $V'=p^V_{\alpha,L}(0)V=p^V_{\alpha,L}(\infty)V$,}
\label{eq:19}
p^{V}_{\alpha,L}(\lambda^{*})=\overline{p^{\overline{V}}_{\alpha^{*},\overline{L}}(\lambda)}.
\end{gather}
Moreover, $\lambda\mapsto p^{V}_{\alpha,L}(\lambda)$ is holomorphic
on $\P^1\setminus\{\pm\alpha\}$. We have
\begin{equation}\label{eq:11}
(p^{V}_{\alpha,L}(0))^{-1}p^{V}_{\alpha,L}(\infty)=\Gamma^L_{\rho_VL}(-1),
\end{equation}
the restriction of which to $V$ is reflection across
$(L\oplus\rho_VL)^{\perp}\cap V$ which has codimension $1$ in $V$
(otherwise $L=\rho_VL$) so that
\begin{equation}
\label{eq:22} \det\bigl(\Gamma^L_{\rho_VL}(-1)\bigr)=-1.
\end{equation}
\end{subequations}

We now have the main result of this section:
\begin{thm}
\label{th:8} Let $(V,d^{\lambda})$ be $k$-perturbed harmonic,
$\alpha\in\C^{\times}\setminus S^1$ and $L$ a $d^{\alpha}$-parallel
null line subbundle with $L,\rho_V L$ complementary.  Moreover, set
$L'=p^{V}_{\alpha,L}(\alpha^{*})\bar{L}$ and
$V'=p^V_{\alpha,L}(0)V$.  Assume also that $L',\rho_{V'}L'$ are
complementary.  Then $r:=p^{V'}_{\alpha^{*},L'}p^V_{\alpha,L}$ is
a dressing gauge for $(V,d^{\lambda})$.
\end{thm}
Thus Theorem \ref{th:1} applies with $r$ so defined to give a
dressing transform of $V$ which, by Theorem~\ref{th:5}, induces a
transformation of constrained Willmore surfaces when $V$ is the
central sphere congruence of such a surface.

\begin{proof}
First $r$ is rational in $\lambda$ on $\P^1$ and holomorphic on
$\P^1\setminus\{\pm\alpha,\pm \alpha^{*}\}$.

Since $d^{\lambda}$ is $\rho_V$-twisted, we see that $\rho_VL$ is
$d^{-\alpha}$-parallel so that, by Lemma~\ref{th:7},
$p^V_{\alpha,L}(\lambda)\cdot d^{\lambda}$ extends from
$\C^{\times}\setminus\{\pm\alpha\}$ to a holomorphic family of
connections $\tilde{d}^{\lambda}$ on $\C^{\times}$.  By the reality
\eqref{eq:5} of $d^{\lambda}$, $\bar{L}$ is
$d^{\alpha^{*}}$-parallel so that $L'$ is
$\tilde{d}^{\alpha^{*}}=p^V_{\alpha,L}(\alpha^{*})\cdot
d^{\alpha^{*}}$-parallel\footnote{Here we use our hypothesis
that $\alpha\notin S^1$.}.  Moreover, we use \eqref{eq:18} to see
that $\tilde{d}^{\lambda}$ is $\rho_{V'}$-twisted so that
$\rho_{V'}L'$ is $\tilde{d}^{-\alpha^{*}}$-parallel.  We
therefore apply Lemma~\ref{th:7} again to see that
$\hd^{\lambda}=r(\lambda)\cdot
d^{\lambda}=p^{V'}_{\alpha^{*},\hat{L}}(\lambda)\cdot\tilde{d}^{\lambda}$
is holomorphic on $\C^{\times}$ with $\hd^1=d$ since $r(1)=1$.

Two applications of \eqref{eq:18} show that
\begin{align*}
r(-\lambda)\circ\rho_V\circ r(\lambda)^{-1}&=
p^{V'}_{\alpha^{*},L'}(-\lambda)\circ
(p^V_{\alpha,L}(-\lambda)\circ\rho_V\circ
p^V_{\alpha,L}(\lambda)^{-1})\circ
p^{V'}_{\alpha^{*},L'}(\lambda)^{-1}\\
&=p^{V'}_{\alpha^{*},L'}(-\lambda)\circ\rho_{V'}\circ
p^{V'}_{\alpha^{*},L'}(\lambda)^{-1}
\end{align*}
is independent of $\lambda$ so that $r$ satisfies \eqref{eq:7}.

Meanwhile, from \eqref{eq:11}, we have
\begin{align*}
r(0)^{-1}r(\infty)&=(p^{V'}_{\alpha^{*},L'}(0)p^{V}_{\alpha,L}(0))^{-1}
p^{V'}_{\alpha^{*},L'}(\infty)p^{V}_{\alpha,L}(\infty)\\
&=p^{V}_{\alpha,L}(0)^{-1}\Gamma^{L'}_{\rho_{V'}L'}(-1)p^{V}_{\alpha,L}(\infty)\\
&=\Gamma^{L''}_{\rho_{V}L''}(-1)\Gamma^{L}_{\rho_{V}L}(-1),
\end{align*}
where $L''=p^{V}_{\alpha,L}(0)^{-1}L'$.  Equation \eqref{eq:22} now
tells us that $\det(r(0)^{-1}r(\infty)_{|V})=1$.

We are therefore left to deal with the reality condition
\eqref{eq:8}.  For this, it suffices to prove that $\lambda\mapsto
\overline{r(\lambda^{*})}r(\lambda)^{-1}$ is holomorphic on
$\P^1$ and so independent of $\lambda$.  The conclusion then follows
since $r(1)=1$.  To stop the decorations piling up, let us
temporarily write $M=L'$, $U=V'$ and $\beta=\alpha^{*}$ so that,
using \eqref{eq:19},
\begin{equation}
\label{eq:23} \overline{r(1/\bar{\lambda})}r(\lambda)^{-1}=
p^{\bar{U}}_{\alpha,\bar{M}}(\lambda)p^V_{\beta,\bar{L}}(\lambda)
p^V_{\alpha,L}(\lambda)^{-1}p^U_{\beta,M}(\lambda)^{-1}.
\end{equation}
This is holomorphic in $\lambda$ on
$\P^1\setminus\{\pm\alpha,\pm\beta\}$.  However, Lemma~\ref{th:11}
tells us that $p^V_{\beta,\bar{L}}
p^V_{\alpha,L}{}^{-1}p^U_{\beta,M}{}^{-1}$ is holomorphic at $\beta$
since $\bar{L}=p^{V}_{\alpha,L}(\beta)^{-1}M$ and holomorphic at
$-\beta$ since
$\rho_{V}\bar{L}=p^{V}_{\alpha,L}(-\beta)^{-1}\rho_{V'}M$.  Since
$p^{\bar{U}}_{\alpha,\bar{M}}$ is also holomorphic at $\pm\beta$, we
see that $\overline{r(\lambda^{*})}r(\lambda)^{-1}$ is
holomorphic there also.

A similar argument with
$p^{\bar{U}}_{\alpha,\bar{M}}p^V_{\beta,\bar{L}}
p^V_{\alpha,L}{}^{-1}$ using $\bar{M}=p^V_{\beta,\bar{L}}(\alpha)L$
establishes holomorphicity at $\pm\alpha$ and we are done.
\end{proof}
\begin{rem}
To construct bundles $L$ that satisfy the hypotheses of
Theorem~\ref{th:8} it suffices to choose an initial condition $L_x$,
some $x\in\Sigma$, that satisfies the (open) complementarity
conditions and extend to an open set by parallel translation since
$d^{\alpha}$ is flat.  We observe that the complementarity conditions
are non-empty since they are satisfied by real, null $L_x$
transverse to $V_x$ and $V_x^{\perp}$.  Indeed, in this case,
$L'_{x}=p^V_{\alpha,L_{x}}(\alpha^{*})L_x=L_x$ and
$\rho_{V'_x}L'_x=\rho_VL_x$ since $L_x$ is an eigenspace of
$p^V_{\alpha,L_x}(\lambda)$, for all $\lambda$.
\end{rem}

By analogy with the Bianchi--B\"acklund transform of constant Gauss
curvature surfaces (c.f. \cite{Pac10}), we make the following
\begin{defn}\label{th:14}
The dressing transform by the dressing gauge $r$ of
Theorem~\ref{th:8} is the \emph{B\"acklund transform with parameters
$\alpha,L$}.
\end{defn}

\begin{rem}
In the classical literature, B\"acklund transforms are specified by
a choice of spectral parameter $\alpha$ and an initial condition
which locates the transformed surface in space at a base-point
$x\in\Sigma$.  In our setting, this last amounts to choosing the
fibre $L_x$ from which our bundle $L$ can be recovered by parallel
translation with respect to the flat connection $d^{\alpha}$.
\end{rem}

\subsection{Bianchi permutability for B\"acklund transforms}
\label{sec:bianchi-perm-backl}

Our terminology in Definition~\ref{th:14} may be justified by the
fact that there is a Bianchi permutability theorem available for our
B\"acklund transform.

Let us recall in outline the statement of Bianchi permutability: we
start with a $k$-perturbed harmonic $(V,d^{\lambda})$ and construct
two B\"acklund transforms $(V_1,d_1^{\lambda})$,
$(V_2,d_2^{\lambda})$ with parameters $\alpha_1,L_1$ and
$\alpha_2,L_2$ respectively.  The theorem then asserts the existence
of a fourth $k$-perturbed harmonic $(V_{12},d_{12})$ which is
simultaneously a B\"acklund transform of $V_1$ with parameter
$\alpha_2$ and of $V_2$ with parameter $\alpha_1$.

This amounts to finding the right line bundles $L_2^1$ and $L_1^2$
with $L_2^1$ $d_1^{\alpha_2}$-parallel and $L_1^2$
$d_2^{\alpha_1}$-parallel.  However, we have natural candidates
close at hand: by construction, we have gauge transformations
$r_1,r_2$ with $r_{j}(\lambda)\cdot d^{\lambda}=d_j^{\lambda}$ and
so
\begin{subequations}
\label{eq:27}
\begin{align}
\label{eq:28}
L_2^1&:=r_{1}(\alpha_2)L_2\\
\label{eq:29} L_{1}^2&:=r_{2}(\alpha_1)L_1
\end{align}
\end{subequations}
fit the bill.  It remains to verify that the B\"acklund
transformations of $V_j$ with parameters $\alpha_k,L^{j}_{k}$,
$j\neq k\in\{1,2\}$, coincide and it is to this that we now turn.

We begin with an effort to keep the notation under control: as
already indicated, we denote by $r_j$ the dressing gauge that
implements the B\"acklund transform of $(V,d^{\lambda})$ with
parameters $\alpha_j,L_j$.  Similarly, let $r_j^k$ be the dressing
gauge that implements the B\"acklund transform of $V_k$ with
parameters $\alpha_j,L_j^k$.  Thus, for example,
\begin{equation*}
r_1(\lambda)=p^{*}_{\alpha_1^{*},L'_1}p^{V}_{\alpha_1,L_1},
\end{equation*}
where $*$ denotes a bundle (in fact $p^{V}_{\alpha_1,L_1}(0)V$) that
we will not have to keep track of.  We have
$d_1^{\lambda}=r_{1}(\lambda)\cdot d^{\lambda}$, off a divisor,
$V_1=r_1(0)V=r_1(\infty)V$, and, in case that $V$ is the central
sphere congruence of a constrained Willmore $\Lambda$,
$\Lambda_1=r_1(0)\Lambda^{1,0}\cap r_1(\infty)\Lambda^{0,1}$.

The key fact is contained in the following lemma:
\begin{lemma}
\label{th:15} Suppose that $\alpha_2\notin\{\pm\alpha_1,\pm
\alpha_1^{*}\}$. Then
\begin{equation}
\label{eq:30} r^1_2r_1=r^2_1r_2.
\end{equation}
\end{lemma}
\begin{proof}
We shall argue as in the proof of Theorem~\ref{th:8} (itself in part
a permutability theorem for $p$'s) and show that
$\Pi:=r^2_1r_2(r_1)^{-1}(r^1_2)^{-1}$ is holomorphic on $\P^1$ and
so constant.  Then \eqref{eq:30} follows by evaluating at
$\lambda=1$.

First we consider holomorphicity at $\alpha_1$: we know that $r^1_2$
is holomorphic and invertible at $\alpha_1$ so the issue is with
$r^2_1r_2(r_1)^{-1}$ which reads
\begin{equation*}
p^{*}_{\alpha_1^{*},\hat{L_1^2}}
p^{V_2}_{\alpha_1,L_1^2}r_2p^V_{\alpha_1,L_1}{}^{-1}
p^{*}_{\alpha_1^{*},\hat{L_1}}{}^{-1}.
\end{equation*}
Here the outer $p^{*}$-terms are holomorphic at $\alpha_1$ (recall
that $\alpha_1\notin S^2$!) and what is left is holomorphic thanks
to Lemma~\ref{th:11} and \eqref{eq:29}.  Thus our product $\Pi$ is
holomorphic and $\rO(\unC^{n+2})$-valued at $\alpha_1$.  The same
argument, with the roles of $\alpha_1$ and $\alpha_2$ swopped shows
that $r^1_2r_1(r_2)^{-1}$ and so $r_2(r_1)^{-1}(r^1_2)^{-1}$ is
holomorphic and invertible at $\alpha_2$.

For the remaining potential singularities at
$-\alpha_j,\pm\alpha_j^{*}$, we observe that, first, $\Pi$
is real: $\overline{\Pi(\lambda)}=\Pi(\lambda^{*})$ and then
that, with $\rho_1^2$ being the reflection across $\Pi_1^2(0)V_2$
and similarly for $\rho_2^1$, we have
\begin{equation*}
\rho_2^1\Pi(\lambda)=\Pi(-\lambda)\rho_1^2.
\end{equation*}
Consequently, holomorphicity at $-\alpha_j,\pm\alpha_j^{*}$
follows from that at $\alpha_j$ and $\Pi$ is holomorphic on $\P^1$
as required.
\end{proof}

With this in hand, the permutability theorem follows quickly: with
$r$ the common value in \eqref{eq:30}, set
$V_{12}=r(0)V=r(\infty)V$, $d^{\lambda}_{12}=r(\lambda)\cdot
d^{\lambda}$ and, in case we are working with a constrained Willmore
surface $\Lambda_{12}=r(0)\Lambda^{1,0}\cap r(\infty)\Lambda^{0,1}$.
Then
\begin{align*}
V_{12}&=r_2^1r_1(0)V=r_2^1(0)V_1 \\
d^{\lambda}_{12}&=r_2^1(\lambda)\cdot(r_1(\lambda)\cdot
d^{\lambda})=r_2^1(\lambda)\cdot d_1^{\lambda}\\
\Lambda_{12}&=\bigl(r_2^1r_1(0)\Lambda^{1,0}\bigr)\cap
\bigl(r_2^1r_1(\infty)\Lambda^{0,1}\bigr)=
r_2^1(0)\Lambda_{1}^{1,0}\cap r_2^1(\infty)\Lambda_{1}^{0,1},
\end{align*}
so that $(V_{12},d_{12}^{\lambda})$ (respectively $\Lambda_{12}$) is
the B\"acklund transform of $(V_1,d_1^{\lambda})$ (respectively
$\Lambda_1$) with parameters $\alpha_2,L_2^1$.  Similarly, this data
is the B\"acklund transform of $(V_2,d_2^{\lambda})$ (respectively
$\Lambda_2$) with parameters $\alpha_1,L_1^2$.  To summarise:
\begin{thm}
\label{th:16} Let $(V,d^{\lambda})$ be $k$-perturbed harmonic and
$\alpha_1,\alpha_2\in\C^{\times}\setminus S^1$ with
$\alpha_2\notin\{\pm\alpha_1,\pm \alpha_1^{*}\}$.  Let
$(V_1,d_1^{\lambda})$, $(V_2,d_2^{\lambda})$ be the B\"acklund
transforms of $(V,d^{\lambda})$ with parameters $\alpha_1,L_1$ and
$\alpha_2,L_2$ respectively.  Then, with $L_j^k$ defined by
\eqref{eq:27}, the B\"acklund transform of $(V_1,d_1^{\lambda})$
with parameters $\alpha_2,L^1_2$ coincides with the B\"acklund
transform of $(V_2,d_2^{\lambda})$ with parameters $\alpha_1,L_1^2$.

Moreover, if $V$ is the central sphere congruence of a constrained
Willmore surface $\Lambda$, the corresponding iterated transforms of
$\Lambda$ coincide also.
\end{thm}

\subsection{Spectral Deformation vs Dressing Transform}
\label{sec:spectr-dress} We conclude our analysis by comparing the
spectral and dressing transformations.  In an appropriate sense,
these commute.  For this, begin with a $k$-perturbed harmonic map
$(V,d^{\lambda})$.  Recall from
section~\ref{sec:spectral-deformation} that, for $\mu\in S^1$, we
have $k$-perturbed harmonic maps $(V_{\mu},d_{\mu}^{\lambda})$ given
by $V_{\mu}=\Phi_{\mu}V$ and $d_{\mu}^{\lambda}=\Phi_{\mu}\cdot
d^{\lambda\mu}$ where
$\Phi_{\mu}\in\Gamma\rO(\underline{\R}^{n+1,1})$ solves
$\Phi_{\mu}\cdot d_{\mu}=d$.  Now let $(\hV,\hd^{\lambda})$ be a
dressing transform of $(V,d^{\lambda})$ via a dressing gauge $r$.
Then $(\hV,\hd^{\lambda})$ has a spectral deformation
$(\hV_{\mu},\hd_{\mu}^{\lambda})$ given by
$\hV_{\mu}=\hat{\Phi}_{\mu}\hV$,
$\hd_{\mu}^{\lambda}=\hat{\Phi}_{\mu}\cdot\hd^{\lambda\mu}$, where
$\hat{\Phi}_{\mu}$ solves $\hat{\Phi}_{\mu}\cdot\hd^{\mu}=d$.

For $\lambda\in\dom(r)\setminus\set{0,\infty}$, we have
$\hd^{\lambda}=r(\lambda)\cdot d^{\lambda}$ so that, defining
$r_{\mu}(\lambda):=\hat{\Phi}_{\mu}r(\lambda\mu)(\Phi_{\mu})^{-1}$,
we have
\begin{align*}
\hd_{\mu}^{\lambda}&=r_{\mu}(\lambda)\cdot
d_{\mu}^{\lambda},\quad\lambda\in\dom(r_{\mu})\setminus\set{0,\infty},\\
\hV_{\mu}&=r_{\mu}(0)V_{\mu}.
\end{align*}
It is now a simple matter to check that $r_{\mu}$ is a dressing
gauge for $V_{\mu}$ and we conclude that $\hV_{\mu}$ is a dressing
transform of $V_{\mu}$.  Thus:
\begin{prop}
Let $(\hV,\hd^{\lambda})$ be a dressing transform of a $k$-perturbed
harmonic $(V,d^{\lambda})$.  For $\mu\in S^1$, the spectral
deformation $(\hV_{\mu},\hd_{\mu}^{\lambda})$ of
$(\hV,\hd^{\lambda})$ is a dressing transform of the spectral
deformation $(V_{\mu},d_{\mu}^{\lambda})$ of $(V,d^{\lambda})$.
\end{prop}

In case that $(\hV,\hd^{\lambda})$ is a B\"acklund transform of
$(V,d^{\lambda})$, we can do a little better.  First observe that, in
this case, $\mu\in\dom(r)\setminus\set{0,\infty}$ so that
$\hd^{\mu}=r(\mu)\cdot d^{\mu}=(r(\mu)\Phi_{\mu}^{-1})\cdot d$ and we
may therefore take $\hat{\Phi}_{\mu}=\Phi_{\mu}r(\mu)^{-1}$.
Further, suppose that the B\"acklund transformation of
$(V,d^{\lambda})$ has parameters $\alpha$, $L$ so that $L$ is a
$d^{\alpha}$-parallel, null line subbundle.  Then $\Phi_{\mu}L$ is
$d_{\mu}^{\mu/\alpha}$-parallel which prompts the following result:
\begin{prop}
Let $(\hV,\hd^{\lambda})$ be the B\"acklund transform of
$(V,d^{\lambda})$ with parameters $\alpha$, $L$.  For $\mu\in S^1$,
the spectral deformation $(\hV_{\mu},\hd_{\mu}^{\lambda})$ is the
B\"acklund transform with parameters $\alpha/\mu$, $\Phi_{\mu}L$ of
the spectral deformation $(V_{\mu},d_{\mu}^{\lambda})$.
\end{prop}
\begin{proof}
In this case, we have
$r(\lambda)=p_{\alpha^{*},L'}^{V'}(\lambda)p_{\alpha,L}^V(\lambda)$.
We set $L_{\mu}=\Phi_{\mu}L$.  Then, in view of the discussion
above, our result amounts to the identity
\begin{equation*}
\Phi_{\mu}r(\mu)^{-1}r(\lambda\mu)\Phi_{\mu}^{-1}=
p^{V_{\mu}'}_{\alpha^{*}/\mu,L_{\mu}'}(\lambda)p^{V_{\mu}}_{\alpha/\mu,L_{\mu}}(\lambda).
\end{equation*}
For this, we use the, by now familiar, permutability argument.
Rearranging all terms onto the left hand side yields a product which
is holomorphic in $\lambda$ except possibly at $\pm\alpha/\mu$ and
$\pm\alpha^{*}/\mu$.  The part with possible singularities at
$\pm\alpha/\mu$ is
\begin{equation*}
\Gamma^L_{\rho_VL}(\tfrac{\lambda\mu-\alpha}{\lambda\mu+\alpha})\Phi_{\mu}^{-1}
(\Gamma^{L_{\mu}}_{\rho_{V_{\mu}}L_{\mu}}(\tfrac{\lambda-\alpha/\mu}{\lambda+\alpha/\mu}))^{-1}
\end{equation*}
which is readily seen to reduce to $\Phi_{\mu}^{-1}$ and so is
independent of $\lambda$.  The reality condition now establishes
holomorphicity at $\pm\alpha^{*}/\mu$ so that the whole product
is independent of $\lambda$ and we are done.
\end{proof}
In particular, we see that any B\"acklund transform can be obtained
as a combination of spectral deformations and a B\"acklund transform
with $\alpha\in\R$.

\section{Codimension $2$}
\label{sec:codimension-2}

Conformal surface geometry in $S^4$ has a distinctive flavour of its
own.  This manifests itself in various ways such as a well-developed
twistor theory \cite{Bry82,EelSal85,Fri88} and the quaternionic
formalism of the Berlin school
\cite{BurFerLesPed02,FerLesPedPin01,Boh10}.  However, for us, the key
feature of this setting is that the space of oriented $2$-spheres is
a complex manifold (in fact, a pseudo-Hermitian symmetric space).  We
will show that this structure allows us both to simplify the
preceding theory and to relate it to a construction of Willmore
surfaces in \cite{BurFerLesPed02}.

\subsection{Additional structure in codimension $2$}
\label{sec:compl-struct-space}

We consider maps of the Riemann surface $\Sigma$ into the space
$\tilde{G}_{3,1}(\R^{5,1})$ of \emph{oriented} $2$-spheres in $S^4$,
or, equivalently, oriented bundles $V$ of $(3,1)$-planes in
$\unC^6$.  Note that if $V$ is the central sphere congruence of a
conformal immersion $\Lambda$, then $\sigma\wedge d_X\sigma\wedge
d_{J^{\Sigma}X}\sigma\wedge \Delta\sigma$, $X\in T\Sigma$,
$\sigma\in\Gamma\Lambda$, gives an orientation of $V$ which is
independent of choices.  

In this setting, $V^{\perp}$ is a real bundle of oriented $(2,0)$-planes
and so splits as a sum of complex conjugate, $\cD^V$-parallel, null
line subbundles:
\begin{equation*}
V^{\perp}=V^{\perp}_{+}\oplus V^{\perp}_{-}.
\end{equation*}
As a consequence, we have a family of $\cD^V$-parallel gauge
transformations
$\tau_V(\lambda):=\Gamma^{V_{+}^{\perp}}_{V_{-}^{\perp}}(\lambda)$ and, in
particular, an almost complex structure $J^V$ on $V\wedge V^{\perp}$
given by $J^V=\tau_V(\sqrt{-1})_{|V\wedge V^{\perp}}$.  We may
therefore split $\cN_V$ into its holomorphic and anti-holomorphic
parts:
\begin{equation}\label{eq:52}
\cN_V=\cA_V+\cQ_V,
\end{equation}
with $\cA_V\circ J^{\Sigma}=J^V\circ\cA_V$ and $\cQ_V\circ
J^{\Sigma}=-J^V\circ\cQ_V$.  Thus $\cA_{V}^{1,0}\in\Omega^{1,0}(V\wedge
V^{\perp}_{+})$ while $\cQ_{V}^{1,0}\in\Omega^{1,0}(V\wedge
V^{\perp}_{-})$.

We note the case where $\cA=0$ (respectively, $\cQ=0$): this means
that $V$ is a (anti-)holomorphic map into the space of oriented
$2$-spheres or, equivalently, $V^{\perp}_{+}$ (respectively,
$V^{\perp}_{-}$) is a holomorphic subbundle of $\unC^{6}$.  If, in
addition, $V$ is the central sphere congruence of a conformal
immersion $\Lambda$, then this condition obtains if and only if
$\Lambda$ is \emph{twistor holomorphic}, that is, $\Lambda$ has a
holomorphic twistor lift to $\C P^3$ \cite[\S 8.1 and
Lemma~22]{BurFerLesPed02}.  Such $\Lambda$ are, of course, Willmore
surfaces.

\subsection{Untwisted family of flat connections}
\label{sec:untwisted-loops-flat}

Let $(V,d^{\lambda})$ be a $k$-perturbed harmonic map into
$\tilde{G}_{3,1}(\R^{5,1})$.  Since $\lambda\mapsto\tau_V(\lambda)$
is a group homomorphism $\C^{\times}\to\rO(\unC^{6})$ and
$\tau_V(-1)=\rho_V$, we readily conclude from \eqref{eq:4} that
\begin{equation}
\label{eq:21}
\tau_V(\lambda)\cdot d^{\lambda}=\tau_V(-\lambda)\cdot d^{-\lambda}.
\end{equation}
Thus there is a second family of connections
$\nabla^{\mu}$, $\mu\in\C^{\times}$, defined by
\begin{equation}
\label{eq:33}
\nabla^{\lambda^2}=\tau_V(\lambda)\cdot d^{\lambda},
\end{equation}
for $\lambda\in\C^{\times}$.

\begin{prop}
\label{th:17}%properties of untwisted
Let $(V,d^{\lambda})$ be a $k$-perturbed harmonic map into
$\tilde{G}_{3,1}(\R^{5,1})$ with $d^{\lambda}=\sum_{|i|\leq k}A_k$.  Define $\nabla^{\mu}$,
$\mu\in\C^{\times}$, by \eqref{eq:33}.  Then:
\begin{subequations}\label{eq:34}
\begin{align}
\label{eq:35}
\nabla^1&=d&&\text{(normalisation)}\\
\label{eq:36}
\nabla^{\mu^{*}}&=\overline{\nabla^{\mu}}&&\text{(reality)}
\end{align}
\end{subequations}
and $\mu\mapsto(\nabla^{\mu})^{1,0}$ is holomorphic on $\C$ with a
pole of order $l$ at $\infty$ (whence, by \eqref{eq:36},
$\mu\mapsto(\nabla^{\mu})^{0,1}$ is holomorphic on
$\C^{\times}\cup\{\infty\}$ with a pole of order $l$ at $0$).  Here
\begin{equation}
\label{eq:37}
l=
\begin{cases}
(k-1)/2&\text{if $k$ is odd and $A_kV_{-}^{\perp}=0$;}\\
k/2&\text{if $k$ is even;}\\
(k+1)/2&\text{otherwise.}
\end{cases}
\end{equation}
Moreover $V_{-}^{\perp}$ is $(\nabla^0)^{1,0}$-stable (whence
$V_{+}^{\perp}$ is $(\nabla^{\infty})^{0,1}$-stable).
\end{prop}
\begin{proof}
We have $\tau_V(1)=1$ and
$\tau_V(\lambda^{*})=\overline{\tau_V(\lambda)}$ which, together
with \eqref{eq:1}, readily yields \eqref{eq:34}.

For $A\in\Omega^1(\Wedge^2\unC^6)$, decompose $A=A^{+}+A^{0}+A^{-}$
according to the eigenspace decomposition of $\Ad_{\tau_V(\lambda)}$:
\begin{equation*}
\Wedge^2\unC^6=(V\wedge V^{\perp}_{+})
\oplus(\Wedge^2V\oplus\Wedge^2V^{\perp})\oplus 
(V\wedge V^{\perp}_{-}).
\end{equation*}
Then 
\begin{equation*}
(\nabla^{\lambda^2})^{1,0}=(A_{0}+A_1^{-})+\lambda^{2}(A_1^{+}+A_2+A_3^{-})+\dots
\end{equation*}
which establishes the holomorphicity of
$\mu\mapsto(\nabla^{\mu})^{1,0}$ on $\C$.  Moreover, as $A_0$
preserves $V_{\pm}^{\perp}$ while $A_1^{-}$ annihilates
$V^{\perp}_{-}$, we conclude that $V^{\perp}_{-}$ is
$(\nabla^0)^{1,0}$-stable.

Finally, we contemplate the leading term in $\lambda$ of
$(\nabla^{\lambda^2})^{1,0}$: if $k$ is even, this is
$\lambda^{k}(A_k+A_{k-1}^{+})$, while, for odd $k$, it is either
$\lambda^{k+1}A^{+}_k$, if $A_k^{+}\neq 0$, or
$\lambda^{k-1}(A_{k-1}+A_k^{-})$ when $A_k^{+}=0$, or, equivalently,
$A_kV^{\perp}_{-}=0$.
\end{proof}

Thus these connections share many properties of the family
$d^{\lambda}$ but are not $\rho_V$-twisted and have poles of half the
order.  In view of the first property, we call $\nabla^{\mu}$ the
\emph{untwisted family of connections} associated to
$(V,d^{\lambda})$.

We can reverse this line of argument and start from an untwisted
family of flat connections:
\begin{thm}
\label{th:18}
Let $V$ be a bundle of $(3,1)$-planes in $\underline{\R}^{5,1}$ and
$\nabla^{\mu}=\sum_{|j|\leq l}\mu^{j}B_{j}$, $\mu\in\C^{\times}$, a family of flat connections
satisfying the conclusions of Proposition~\ref{th:17}.  

Define connections $d^{\lambda}$ by
\begin{equation}
\label{eq:39}
d^{\lambda}=\tau_V(\lambda)^{-1}\cdot\nabla^{\lambda^2},
\end{equation}
$\lambda\in\C^{\times}$.  Then $d^{\lambda}$ satisfies the conditions
of Definition \ref{th:10} so that $(V,d^{\lambda})$ is a
$k$-perturbed harmonic bundle where
\begin{equation}\label{eq:40}
k=
\begin{cases}
2l-1&\text{if
$B_{l}^{-}=B_l^0=0$;}\\
2l&\text{if $B_l^{-}=0$ but $B_l^0\neq 0$;}\\
2l+1&\text{otherwise.}
\end{cases}
\end{equation}
\end{thm}
\begin{proof}
It is clear that, with $d^{\lambda}$ defined by \eqref{eq:39}, we
have $d^1=d$ and that the reality condition \eqref{eq:5} holds.
Moreover, we readily compute:
\begin{equation*}
d^{-\lambda}=\tau_V(-\lambda)^{-1}\cdot\nabla^{\lambda^2}=
(\rho_V\tau_V(\lambda)^{-1})\cdot\nabla^{\lambda^{2}}=\rho_V\cdot d^{\lambda},
\end{equation*}
so that the twisting condition \eqref{eq:4} holds also.

As for the dependence in $\lambda$ of $d^{\lambda}$, it is clear that
$\lambda\mapsto d^{\lambda}$ is holomorphic on $\C^{\times}$ while
the argument of Lemma~\ref{th:7} can be easily adapted to treat the
family of partial connections
$\lambda\mapsto(\nabla^{\lambda^{2}})^{1,0}$ and we deduce that
$(d^{\lambda})^{1,0}$ is holomorphic at $0$ since $V^{\perp}_{-}$ is
$(\nabla^0)^{1,0}$-stable.  Finally, the leading terms in $\lambda$ of
$d^{\lambda}$ are
\begin{equation*}
\lambda^{2k+1}B^{-}_l+\lambda^{2k}B_l^0+\lambda^{2k-1}(B_l^{+}+B_{l-1}^{-})+\dots
\end{equation*}
from which \eqref{eq:40} follows at once.
\end{proof}

We conclude that the families $d^{\lambda}$ and $\nabla^{\mu}$ carry
the same information and therefore refer to a $k$-perturbed harmonic
bundle in codimension $2$ as either $(V,d^{\lambda})$ or
$(V,\nabla^{\mu})$, where $d^{\lambda}$ and $\nabla^{\mu}$ are
related via \eqref{eq:33} and \eqref{eq:39}.

We list some examples:
\begin{enumerate}
\item $\nabla^{\mu}\equiv d$ satisfies the conditions of
Proposition~\ref{th:17} if and only if $V^{\perp}_{+}$ is
holomorphic.  This is the case $l=0$ of Theorem \ref{th:18}.
\item $V$ is harmonic if and only if
$\nabla^{\mu}=d+(\mu-1)\cA^{1,0}+(\mu^{-1}-1)\cA^{0,1}$ is flat, for
all $\mu\in\C^{\times}$.  A similar result is true for maps into any
(pseudo-)Hermitian symmetric space.
\item Let $V$ be the central sphere congruence of $\Lambda$.  Then
$\Lambda$ is constrained Willmore with multiplier $q$ if and only if
$\nabla^{\mu}=d+(\mu-1)(q+\cA)^{1,0}+(\mu^{-1}-1)(q+\cA)^{0,1}$ is
flat, for all $\mu\in\C^{\times}$ \cite{Boh10}.  This is the case
$l=1$ of Theorem \ref{th:18} with $B_1^{-}=0$ and the additional
constraint that $B_1^0$ take values in $\Lambda\wedge\Lambda^{(1)}$.
\end{enumerate}

\subsection{Untwisted dressing gauges}
\label{sec:untw-dress-gaug}

The untwisted family of connections has a simpler algebraic structure. It
is therefore unsurprising that we may dress such families with
algebraically simpler gauge transformations:

\begin{defn}\label{th:19}
Let $(V,\nabla^{\mu})$ be a $k$-perturbed harmonic bundle in
$\unC^{6}$.  An \emph{untwisted dressing gauge} for
$(V,\nabla^{\mu})$ is a family of gauge transformations $\mu\mapsto
R(\mu)\in\Gamma\rO(\unC^6)$ which is holomorphic in $\mu$ near
$0,\infty\in\P^1$ and has the following properties:
\begin{enumerate}
\item[(a)]\label{item:10} For all $\mu\in\dom(R)$,
\begin{equation}
\label{eq:38}
R(\mu^{*})=\overline{R(\mu)}\qquad\text{(reality)}.
\end{equation}
\item[(b)]\label{item:11} The connections $R(\mu)\cdot\nabla^{\mu}$
extend from $\dom(R)\setminus\set{0,\infty}$ to a holomorphic family
of connections $\hat{\nabla}^{\mu}$ on $\C^{\times}$ with
$\hat{\nabla}^1=d$.
\item[(c)]\label{item:12} The line bundles $R(0)V_{-}^{\perp}$ and
$R(\infty)V_{+}^{\perp}$ are complementary.
\end{enumerate}
For such an $R$, set $\hat{V}=(R(0)V_{-}^{\perp}\oplus
R(\infty)V_{+}^{\perp})^{\perp}$, oriented so that
\begin{align}
\label{eq:42}
\hV^{\perp}_{-}&=R(0)V^{\perp}_{-},&\hV^{\perp}_{+}&=R(\infty)V_{+}^{\perp},
\end{align}
and call $(\hV,\hat{\nabla}^{\mu})$ the \emph{untwisted dressing
transform of $(V,\nabla^{\mu})$ by $R$}.
\end{defn}

We note that $\hV$ is a bundle of $(3,1)$-planes since $\hV^{\perp}$
is a sum of complex conjugate null lines.

It is our contention that $(\hV,\hat{\nabla}^{\mu})$ is again
$k$-perturbed harmonic.  This will follow at once from the next
proposition which reduces the situation to the twisted setup of
section \ref{dressingaction}.
\begin{prop}\label{th:21}
Let $(\hV,\hat\nabla^{\mu})$ be the untwisted dressing transform of
$(V,\nabla^{\mu})$ by $R$.  Set
\begin{equation}
\label{eq:41}
r(\lambda)=\tau_{\hV}(\lambda)^{-1}\circ R(\lambda^2)\circ\tau_V(\lambda).
\end{equation}
Then $r$ is a dressing gauge for $(V,d^{\lambda})$ and, with
$\hd^{\lambda}=\tau_{\hV}(\lambda)^{-1}\cdot\hat{\nabla}^{\lambda^2}$,
$(\hV,\hd^{\lambda})$ is the dressing transform of $(V,d^{\lambda})$
by $r$.

In particular, $(V,d^{\lambda})$ is $k$-perturbed harmonic by
Theorem~\ref{th:1}.
\end{prop}
\begin{proof}
We check that $r$ defined by \eqref{eq:41} satisfies the conditions
of Definition~\ref{th:9}.  For this, first note that
Lemma~\ref{th:11} along with \eqref{eq:42}, shows that
$\lambda\mapsto r(\lambda)$ is holomorphic near $0$ and $\infty$.
Now, for item (\ref{item:1}) of Definition~\ref{th:9}, a short
calculation gives
\begin{equation}
\label{eq:43}
r(-\lambda)\circ\rho_V\circ r(\lambda)^{-1}=\rho_{\hV},
\end{equation}
which is certainly independent of $\lambda$. Item (\ref{item:2}) is
immediate from the corresponding reality conditions on
$R,\tau_V,\tau_{\hV}$.  In particular, evaluating at $\lambda=0$
gives $r(\infty)=\overline{r(0)}$.  Again, with
$d^{\lambda}=\tau_V(\lambda)^{-1}\cdot\nabla^{\lambda^2}$, we have,
from item (b) of Definition~\ref{th:19}, that
$\hd^{\lambda}=r(\lambda)\cdot d^{\lambda}$ on
$\dom(r)\setminus\set{0,\infty}$.  However, it is clear from its
definition that $\lambda\mapsto\hd^{\lambda}$ is holomorphic on
$\C^{\times}$ and that $\hd^1=d$.  This settles item (\ref{item:3}).
Finally, for item (\ref{item:4}), evaluate \eqref{eq:43} at
$\lambda=0,\infty$ to deduce that $\hV=r(0)V=r(\infty)V$.  We have
already noted that $\hV$ is a bundle of $(3,1)$-planes and now
Lemma~\ref{th:22}(\ref{item:8}) together with
$r(\infty)=\overline{r(0)}$ immediately yields \eqref{eq:31}.

We therefore conclude that $r$ is indeed a dressing gauge for
$(V,d^{\lambda})$.  Moreover, since $r(\lambda)\cdot
d^{\lambda}=\hd^{\lambda}$, for
$\lambda\in\dom(r)\setminus\set{0,\infty}$, and $\hV=r(0)V$, we see
that $(\hV,\hd^{\lambda})$ is the dressing transform of
$(V,d^{\lambda})$ by $r$.
\end{proof}

With an eye to dressing constrained Willmore surfaces, we compute
$r(0)$ and $r(\infty)$:
\begin{lemma}
\label{th:20}
With $V,\hV$ and $r,R$ as in Proposition~\ref{th:21}, define
projections $\pi_{i},\hat{\pi}_{i}$ by 
\begin{align*}
\tau_V(\lambda)&=\lambda\pi_1+\pi_0+\lambda^{-1}\pi_{-1}\\
\tau_{\hV}(\lambda)&=\lambda\hat{\pi}_1+\hat{\pi}_0+\lambda^{-1}\hat{\pi}_{-1}.
\end{align*}
Then:
\begin{align}
\label{eq:44}
r(0)&=\sum_{|i|\leq 1}\hat{\pi}_iR(0)\pi_{i}&
r(\infty)&=\sum_{|i|\leq 1}\hat{\pi}_iR(\infty)\pi_{i}.
\end{align}
\end{lemma}
\begin{proof}
We have already seen that $r(0)V=r(\infty)V=\hV$ but more is true: rearrange
\eqref{eq:41} to give
\begin{equation*}
R(\lambda^2)=\tau_{\hV}(\lambda)\circ r(\lambda)\circ\tau_V(\lambda)^{-1},
\end{equation*}
so that the holomorphicity of $\lambda\mapsto R(\lambda^2)$ near $0$
and $\infty$, along with Lemma~\ref{th:11}, tells us that
$r(0)V_{+}^{\perp}=\hV_{+}^{\perp}$ and
$r(\infty)V_{-}^{\perp}=\hV_{-}^{\perp}$.  Moreover, since $r(0)$ is
an isometry and $V_{-}^{\perp}$ is maximal
isotropic in $V^{\perp}$, we have that $r(0)V_{-}^{\perp}$ is maximal
isotropic in $\hV^{\perp}$ and so must be $\hV_{-}^{\perp}$.  Thus
\begin{equation*}
r(0)=\sum_{|i|\leq 1}\hat{\pi}_ir(0)\pi_{i}
\end{equation*}
and similarly for $r(\infty)$.

However,
\begin{equation*}
\sum_{|i|\leq 1}\hat{\pi}_ir(\lambda)\pi_{i}=
\sum_{|i|\leq 1}\hat{\pi}_i\tau_{\hV}(\lambda)^{-1}
R(\lambda^2)\tau_V(\lambda)\pi_{i}
=\sum_{|i|\leq 1}\hat{\pi}_i
R(\lambda^2)\pi_{i}
\end{equation*}
and evaluating this last at $\lambda=0,\infty$ establishes
\eqref{eq:44}.
\end{proof}

Putting all this together with Theorem \ref{th:5} yields:
\begin{thm}
Let $(V,\nabla^{\mu})$ be the $2$-perturbed harmonic central sphere
congruence of a constrained Willmore surface $\Lambda$ with multiplier
$q$.  Let $(\hV,\hat{\nabla}^{\mu})$ be the untwisted dressing
transform of $(V,\nabla^{\mu})$ by $R$.

Define:
\begin{subequations}
\begin{align}
\label{eq:46}
\hL&=(\pi_{\hV}R(0)\Lambda^{1,0})\cap(\pi_{\hV}R(\infty)\Lambda^{0,1})\\
\label{eq:47}
\hq&=\Ad_{\pi_{\hV}R(\infty)}q^{1,0}+\Ad_{\pi_{\hV}R(0)}q^{0,1}.
\end{align}
\end{subequations}
Then $\hL$ is a constrained Willmore surface with multiplier $\hq$
and central sphere congruence $\hV$ on the open set where it
immerses.
\end{thm}
\begin{defn}
We say that $\hL$ is the \emph{untwisted dressing transform} of
$\Lambda$ by $R$.
\end{defn}
We shall see below in Proposition~\ref{th:28} and
Corollary~\ref{th:26} that a quaternionic formalism provides cleaner
formulae in which the projections $\pi_{\hV}$ do not intervene.

\begin{rem}
Which dressing gauges $r$ arise from untwisted dressing gauges $R$
via \eqref{eq:41}?  Lemma~\ref{th:20} provides a necessary condition:
$r(0)V_{\pm}^{\perp}=\hV_{\pm}^{\perp}$ and
$r(\infty)V_{\pm}^{\perp}=\hV_{\pm}^{\perp}$.  It is not difficult to
see that this condition is also sufficient by arguing as in
Proposition~\ref{th:21}.

Any dressing gauge has $r(0)V^{\perp}=\hV^{\perp}$ and
$r(0)V_{\pm}^{\perp}$ maximal isotropic in $\hV^{\perp}$ so that the
remaining possibility is that $r(0)V_{\pm}^{\perp}=\hV_{\mp}^{\perp}$ and
$r(\infty)V_{\pm}^{\perp}=\hV_{\mp}^{\perp}$.  This amounts to switching
the orientation of $\hV^{\perp}$ or, equivalently, replacing
$\tau_{\hV}$ by its inverse.  We leave it as an exercise to the
interested reader to work out the analogue of our analysis in this
case.
\end{rem}

\subsection{Untwisted simple factors}
\label{sec:untw-dress-gaug-1}

The lack of $\rho_V$-twisting allows the possibility of untwisted
dressing gauges with just two poles so that, in codimension $2$, we
may dress by simple factors.  For this, let
$(V,d^{\lambda})=(V,\nabla^{\mu})$ be a $k$-perturbed harmonic
bundle.  For $\nu\in\C^{\times}\setminus S^{1}$, let $M$ be a null,
line subbundle such that $M$ and $\overline{M}$ are complementary on
an open set\footnote{This is the open set on which $M$ is not real.}.
Use this data to define gauge transformations $P_{\nu,M}(\mu)$ on
that set by
\begin{equation*}
P_{\nu,M}(\mu)=\Gamma_{\overline{M}}^M\bigl(
\tfrac{(1-\nu^{*})(\mu-\nu)}
{(1-\nu)(\mu-\nu^{*})}\bigr).
\end{equation*}
We have:
\begin{subequations}
\begin{gather}
\label{eq:49}
P_{\nu,M}(1)=1,\\
\label{eq:50}
P_{\nu,M}(\mu^{*})=\overline{P_{\nu,M}(\mu)},
\end{gather}
\end{subequations}
and that $\mu\mapsto P_{\nu,M}(\mu)$ is rational on $\P^1$ and
holomorphic on $\P^1\setminus\set{\nu,\nu^{*}}$.

\begin{thm}
\label{th:23}
Let $(V,\nabla^{\mu})$ be a $k$-perturbed harmonic bundle in $\unC^{6}$,
$\nu\in\C^{\times}\setminus S^{1}$ and $M$ a $\nabla^{\nu}$-parallel
null line subbundle of $\unC^6$.  Assume that $M,\overline{M}$ are
complementary and then that $P_{\nu,M}(\infty)V^{\perp}_{+}$ and
$P_{\nu,M}(0)V^{\perp}_{-}$ are complementary.  Then
$P_{\nu,M}$ is an untwisted dressing gauge for $(V,\nabla^{\mu})$.
\end{thm}
\begin{proof}
We have already seen that $\mu\mapsto P_{\nu,M}(\mu)$ is holomorphic
on $\P^1\setminus\set{\nu,\nu^{*}}$ and so, in particular, near
$0$ and $\infty$.  Moreover, \eqref{eq:50} is precisely the reality
condition \eqref{eq:38} of Definition~\ref{th:19}.  

We have that $M$ is $\nabla^{\nu}$-parallel, whence, thanks to
\eqref{eq:50}, $\overline{M}$ is $\nabla^{\nu^{*}}$-parallel.  It
is now immediate from Lemma~\ref{th:7} that $\mu\mapsto
P_{\nu,M}(\mu)\cdot\nabla^{\mu}$ extends holomorphically to
$\C^{\times}$ settling Definition~\ref{th:19}(b).

Finally,  Definition~\ref{th:19}(c) holds by hypothesis.
\end{proof}

In fact, the dressing transformation induced by these simple factor
dressing gauges are generically B\"acklund transforms as in
Section~\ref{sec:backl-transf}.  Indeed, if $M$ is
$\nabla^{\nu}$-parallel and $\alpha=\pm\sqrt{\nu}$, then
\eqref{eq:39} tells us that $L:=\tau_V(\alpha)^{-1}M$ is
$d^{\alpha}$-parallel.  Thus $\alpha,L$ are parameters for a
B\"acklund transform of $(V,d^{\lambda})$ which turns out to have the
same effect as dressing by $P_{\nu,M}$.  In more detail:
\begin{prop}\label{th:24}
Let $(V,\nabla^{\mu})$ be a $k$-perturbed harmonic bundle in
$\unC^6$, $\nu\in\C^{\times}\setminus S^1$ and $M$ a
$\nabla^{\nu}$-parallel null line subbundle satisfying the hypotheses
of Theorem~\ref{th:23}.  Let $(\hV,\hat{\nabla}^{\mu})$ be the
dressing transform of $(V,\nabla^{\mu})$ by $P_{\nu,M}$.

Let $\alpha=\pm\sqrt{\nu}$ and set $L=\tau_V(\alpha)^{-1}M$.  Assume
that $L$ satisfies the hypotheses of Theorem~\ref{th:8}: thus
$L,\rho_V L$ are complementary and, with
$L'=p^{V}_{\alpha,L}(\alpha^{*})\bar{L}$ and
$V'=p^V_{\alpha,L}(0)V$, $L',\rho_{V'}L'$ are also complementary.

Then
\begin{equation}
\label{eq:45}
p^{V'}_{\alpha^{*},L'}(\lambda)p^V_{\alpha,L}(\lambda)=
\tau_{\hV}(\lambda)^{-1}P_{\nu,M}(\lambda^2)\tau_V(\lambda),
\end{equation}
for all $\lambda$.
\end{prop}
Thus, Proposition~\ref{th:21} tell us that $(\hV,\hd^{\lambda})$ is
the B\"acklund transform of $(V,d^{\lambda})$ with parameters
$\alpha,L$.
\begin{proof}
Set
\begin{align*}
L(\lambda)&:=p^{V'}_{\alpha^{*},L'}(\lambda)p^V_{\alpha,L}(\lambda)\\
R(\lambda)&:=\tau_{\hV}(\lambda)^{-1}P_{\nu,M}(\lambda^2)\tau_V(\lambda).
\end{align*}
We follow what should now be a familiar strategy: we use
Lemma~\ref{th:11} to see that $\lambda\mapsto
R(\lambda)L(\lambda)^{-1}$ is holomorphic on $\P^1$ and so constant.

Both $L(\lambda)$ and $R(\lambda)$ are holomorphic on
$\P^1\setminus{\pm\alpha,\pm\alpha^{*}}$, $L(\lambda)$ by construction and
$R(\lambda)$ by Proposition~\ref{th:21}.  Moreover, both $R,L$
satisfy the reality condition so it suffices to show holomorphicity
of $RL^{-1}$ at $\pm\alpha$.  Now the part of this product with poles
at $\pm\alpha$ is
\begin{subequations}
\begin{align}
\label{eq:51}
\Gamma^M_{\overline{M}}
\bigl(\tfrac{\lambda^2-\alpha^2}{\lambda^2-1/\bar{\alpha}^2}\bigr)
\tau_{V}(\lambda)
\Gamma_{\rho_V
L}^L\bigl(\tfrac{\lambda-\alpha}{\lambda+\alpha}\bigr)^{-1}&=
\Gamma^M_{\overline{M}}
\bigl(\tfrac{\lambda+\alpha}{\lambda+\alpha^{*}}\bigr)
\Gamma^M_{\overline{M}}
\bigl(\tfrac{\lambda-\alpha}{\lambda-\alpha^{*}}\bigr)
\tau_{V}(\lambda)
\Gamma_{\rho_V
L}^L\bigl(\tfrac{\lambda-\alpha}{\lambda+\alpha}\bigr)^{-1}\\
&=
\Gamma^M_{\overline{M}}
\bigl(\tfrac{\lambda-\alpha}{\lambda-\alpha^{*}}\bigr)
\Gamma^M_{\overline{M}}
\bigl(\tfrac{\lambda+\alpha}{\lambda+\alpha^{*}}\bigr)
\tau_{V}(\lambda)
\Gamma^{\rho_V
L}_{L}\bigl(\tfrac{\lambda+\alpha}{\lambda-\alpha}\bigr)^{-1}.
\label{eq:53}
\end{align}
\end{subequations}
We have $M=\tau_V(\alpha)L$ and this, together with
Lemma~\ref{th:11} applied to the right hand side of \eqref{eq:51},
gives holomorphicity at $\alpha$.  Again,
$\tau_V(-\alpha)\rho_{V}L=\tau_V(\alpha)L=M$ so that
Lemma~\ref{th:11} applies to \eqref{eq:53} to give holomorphicity at
$-\alpha$.
\end{proof}

\subsection{Quaternionic formalism}
\label{sec:riccati-equation}

The Klein correspondence offers another viewpoint on the conformal
geometry of $S^4$ which has been heavily exploited by Pedit, Pinkall
and their collaborators
\cite{Boh10,BurFerLesPed02,FerLesPedPin01,BohPedPin09a}.  We rehearse
the basics of this viewpoint with a view to extending the
transformation theory of \cite[Chapter~12]{BurFerLesPed02} and
relating it to our dressing transformations in
section~\ref{sec:riccati-equation-1}.

We begin with a $2$-dimensional quaternionic vector space which we
view as $\C^4$ equipped with a quaternionic structure $j$: thus
$j:\C^4\to\C^4$ is anti-linear and $j^2=-1$.  Now fix
$\det\in\Wedge^{2}(\C^4)^{*}$ with
\begin{equation*}
j^{*}\det=\overline{\det},\qquad \det(v_1\wedge v_2\wedge jv_1\wedge jv_2)>0,
\end{equation*}
when $v_1\wedge v_2\wedge jv_1\wedge jv_2\neq 0$ $v_i\in\C^4$.  This
data equips $\C^6:=\Wedge^2\C^4$ with a real structure
$\overline{v_1\wedge v_2}:=jv_1\wedge jv_2$ and an inner product
$(\xi,\eta):=\det(\xi\wedge\eta)$ of signature $(5,1)$.

Set $\rSL(\H^2)=\set{g\in\rSL(\C^4):jg=gj}$.  The action of
$\rSL(\H^2)$ on $\Wedge^2\C^{4}$ induces a map
$\rSL(\H^2)\to\rSO(5,1)$ which is a double cover.  Differentiating
gives an isomorphism of Lie algebras $\fsl(\H^2)\cong\fo(\R^{5,1})$
where $\fsl(\H^2)=\set{A\in\fsl(\C^4):[A,j]=0}$.

The Klein correspondence identifies the Grassmannian of $2$-planes in
$\C^4$ with the quadric in $\P(\C^6)$ defined by our inner product:
$W\mapsto\Wedge^2 W$.  Clearly, $j$-stable $2$-planes (thus
$1$-dimensional quaternionic subspaces) are identified with points of
$\P(\cL)=S^4$ yielding the celebrated isomorphism $\H P^1\cong S^4$.

Under this correspondence, oriented $2$-spheres are identified with
$2$-planes $S^{+}\leq\C^4$ which are not $j$-stable: for such an
$S^{+}$, set $S^{-}=jS^{+}$ and then $\C^4=S^{+}\oplus S^{-}$ and the
corresponding $(3,1)$-plane $V$ in $\R^{5,1}$ is given by:
\begin{equation*}
V=S^{+}\wedge S^{-},\qquad
V^{\perp}_{\pm}=\Wedge^2 S^{\pm}.
\end{equation*}
Equivalently, such $S^{+}$ correspond bijectively to $S\in\rSL(\H^2)$
with $S^2=-1$ via $S=\pm i$ on $S^{\pm}$.

Thus, a bundle $V$ of oriented $(3,1)$-planes amounts to
$S\in\Gamma\rSL(\underline{\H}^2) $ with $S^2=-1$ and $\pm
i$-eigenbundles $S^{\pm}$ so that $V^{\perp}_{\pm}=\Wedge^2S^{\pm}$.
The corresponding decomposition $\unC^4=S^{+}\oplus S^{-1}$ induces a
decomposition of the flat connection
\begin{equation*}
d=\cD_S+\cN_{S}
\end{equation*}
with $\cD_SS=0$ and $\cN_S$ anti-commuting with $S$:
$\{\cN_S,S\}=0$.  Using $V=S^{+}\wedge S^{-}$, we readily see that
$\cD_S$ induces the connection $\cD_V$ on $\unC^6=\Wedge^2\unC^4$ and
that $\cN_S$ coincides with $\cN_V$ under the isomorphism
$\fsl(\H^2)\cong\fo(5,1)$.  Again, the subspace $V\wedge V^{\perp}$
of $\fo(5,1)$ corresponds to $\set{A\in\fsl(\H^2):\{A,S\}=0}$ with
the almost complex structure $J_V$ corresponding to post-composition
with $S$.  Thus the decomposition \eqref{eq:52} of $\cN_V$ is
identified with the decomposition
\begin{equation*}
\cN_S=\cA_S+\cQ_S
\end{equation*}
of \cite[\S5.1]{BurFerLesPed02} where
\begin{subequations}\label{eq:68}
\begin{align}
\label{eq:54}
*\cA_S&=S\cA_S=-\cA_SS\\
\label{eq:55}
*\cQ_S&=-S\cQ_S=\cQ_SS.
\end{align}
\end{subequations}
Here, and below, we follow \cite{BurFerLesPed02} by setting $*\alpha=\alpha\circ
J^{\Sigma}$, for $\alpha\in\Omega^1_{\Sigma}$.  Thus our $*$ is minus
that of Hodge.

We note the following simple consequences of this analysis
\cite[equations (5.2) and (5.11)]{BurFerLesPed02}:
\begin{subequations}
\begin{align}
\label{eq:56}
dS&=2(*\cQ_S-*\cA_S)\\
\label{eq:57}
0&=\cA_S\wedge\cQ_S\\
\label{eq:58}
0&=\cQ_S\wedge\cA_S,
\end{align}
\end{subequations}
where coefficients in the wedge products are multiplied using
composition in $\End(\C^4)$.

Now let $\Lambda$ be a conformal immersion with oriented central
sphere congruence $V$.  Then $\Lambda$ corresponds to a quaternionic
line subbundle of $\unC^4$, thus a $j$-stable, rank $2$ bundle
$L\leq\unC^4$ with $\Wedge^2L=\Lambda$.  To the central sphere
congruence $V$ corresponds $S\in\Gamma\rSL(\underline{\H}^2)$
characterised by the following conditions \cite[\S5.2,
Theorem~2]{BurFerLesPed02}:
\begin{subequations}
\begin{gather}
\label{eq:59}
SL=L,\quad dSL\leq T^{*}\Sigma\otimes L\\
\label{eq:60}
*\delta=S\delta=\delta S\\
\cQ_{S}L=0,\label{eq:61}
\end{gather}
\end{subequations}
where $\delta=\pi\circ d$ for $\pi:\unC^4\to\unC^4/L$ the projection
away from $L$. %Remark about orientation here

From \eqref{eq:59}, we see that $L=L^{+}\oplus L^{-}$ where
$L^{\pm}=L\cap S^{\pm}$ and $L^{-}=j L^{+}$.  Thus
$\Lambda=L^{+}\wedge L^{-}$.  We remark that $L^{\pm}$, viewed as
maps $\Sigma\to\C P^3$ are the two twistor lifts of $\Lambda$.
Moreover, from \eqref{eq:60}, we readily compute that:
\begin{equation*}
\Lambda^{1,0}=S^{+}\wedge L^{-},\qquad \Lambda^{0,1}=L^{+}\wedge S^{-}.
\end{equation*}

We have therefore established a dictionary between subbundles of
$\unC^4$ and $\unC^6=\Wedge^2\unC^4$ which we shall use without
further comment.  Thus we can (and will!) talk of conformal
immersions $L$ with central sphere congruence $S$.

With a view to working with constrained Willmore surfaces in this
formalism, we list the algebraic properties of a Lagrange multiplier
$q$ when viewed as an $\fsl(\H^2)$-valued $1$-form.
\begin{lemma}
\label{th:25}
Let $\Lambda=\Wedge^{2}L$ be a conformal immersion with central sphere congruence
$S$.  Let $q\in\Omega^1\otimes\fo(5,1)$.
Then $q^{1,0}$ takes values in $\Lambda\wedge\Lambda^{0,1}$ if and
only if, when viewed as a $\fsl(\H^2)$-valued $1$-form, we have
\begin{subequations}\label{eq:66}
\begin{gather}
\label{eq:64}
q\in\Omega^1(\End(\H^2/L,L))\\
*q=Sq=qS.\label{eq:65}
\end{gather}
\end{subequations}
\end{lemma}
\begin{proof}
Suppose that $q^{1,0}$ takes values in $\Lambda\wedge\Lambda^{0,1}$,
or, equivalently, $q$ preserves $V$ and annihilates
$V^{\perp}$ while $q^{1,0}$ additionally annihilates $\Lambda^{0,1}$.
The first two of these imply
that $qS=Sq$ so that $q$ preserves $S^{\pm}$.
Moreover, $qV^{\perp}=0$ means that $q\Wedge^2S^{\pm}=0$ so
that $\trace(q_{|S^{\pm}})=0$.

Let $Z\in T^{1,0}\Sigma$, $\sigma^{+}\in L^{+}$ and $s^{-}\in
S^{-}$.  Then 
\begin{equation*}
q_{Z}(\sigma^{+}\wedge s^{-})=
(q_{Z}\sigma^{+})\wedge s^{-}+\sigma^{+}\wedge (q_{Z}s^{-})=0.
\end{equation*}
It follows at once that $L^{+}$ is an eigenspace of $q_Z$ with
eigenvalue $\lambda$, say, and then that $q_{Z}=-\lambda$ on
$S^{-}$.  Since $q_Z$ is trace-free on $S^{-}$, we immediately get
that $q_Z$ annihilates $S^{-}$ (so that $*q=qS$) and also $L^{+}$ and
so $L$.  Finally, choose $s^{+}\in S^{+}$ so that $\sigma^{+}\wedge
s^{+}\neq 0$.  Then $0=q_Z(\sigma^{+}\wedge s^{+})=\sigma^{+}\wedge
(q_Zs^{+})$ so that $q_ZS^{+}\leq L^{+}$.  This establishes
equations \eqref{eq:66}.  The converse is straightforward.
\end{proof}

Let $R$ be a dressing gauge for a $k$-perturbed harmonic
$(S,\nabla^{\mu})$.  We wish to describe its effect in quaternionic
terms for which we need a little notation: for $g\in\rO(\C^6)$, write
$\widetilde{g}$ for a preimage of $g$ under the double covering
$\rSL(\C^4)\to\rO(\C^6)$.  Thus
\begin{equation*}
g(v\wedge w)=(\widetilde{g}v)\wedge(\widetilde{g}w),
\end{equation*}
for $v,w\in\C^{4}$.  Clearly $\widetilde{g}$ is determined up to sign
by $g$ and has an unambiguous projective action on $\C^4$.

With this in hand, we have:
\begin{prop}\label{th:28}
Let $(S,\nabla^{\mu})$ be $k$-perturbed harmonic in $\unC^{6}$.  Let $R$ be an
untwisted dressing gauge for $(S,\nabla^{\mu})$ and
$(\hS,\hat{\nabla}^{\mu})$ the untwisted dressing transform by $R$.  Then
\begin{subequations}\label{eq:71}
\begin{align}\label{eq:69}
\hat{S}^{+}&=\widetilde{R}(\infty)S^{+}&\hat{S}^{-}&=\widetilde{R}(0)S^{-}
\end{align}
Moreover, if $S$ is the central sphere congruence of a constrained
Willmore surface $L$ with multiplier $q$, then the dressing transform
$\hS$ and its
multiplier $\hat{q}$ are given by:
\begin{align}\label{eq:70}
\hat{L}^{+}&=\widetilde{R}(\infty)L^{+}&\hat{L}^{-}&=\widetilde{R}(0)S^{-}\\
\label{eq:74}\hq_{|\hat{S}^{+}}
&=\widetilde{R}(\infty) q\widetilde{R}(\infty)^{-1}_{|\hat{S}^{+}}&
\hq_{|\hat{S}^{-}}&=\widetilde{R}(0) q\widetilde{R}(0)^{-1}_{|\hat{S}^{-}}
\end{align}
\end{subequations}
\end{prop}
\begin{proof}
First note that the right members of \eqref{eq:71} follow at once
from the left members because $R(0)=\overline{R}(\infty)$ so that we
may take $\widetilde{R}(0)$ to be $j\circ\widetilde{R}(\infty)\circ
j^{-1}$.

For \eqref{eq:69}, first note that \eqref{eq:42} amounts to
$\Wedge^2\hat{S}^{+}=R(\infty)\Wedge^2 S^{+}$ which immediately
yields $\hat{S}^{+}=\widetilde{R}(\infty)S^{+}$.

For the rest, let $\pi^{\pm}$, $\hat{\pi}^{\pm}$ be the projections
corresponding to the decompositions $\underline{\C}^4=S^{+}\oplus
S^{-}$ and $\underline{\C}^{4}=\hat{S}^{+}\oplus \hat{S}^{-}$.  It is
easy to deduce from Lemma~\ref{th:20} that
\begin{align*}
\widetilde{r}(\infty)&=
\hat{\pi}^{+}\widetilde{R}(\infty)\pi^{+}+\hat{\pi}^{-}\widetilde{R}(\infty)\pi^{-}=
\widetilde{R}(\infty)\pi^{+}+\hat{\pi}^{-}\widetilde{R}(\infty)\pi^{-}\\
\widetilde{r}(0)&=
\hat{\pi}^{+}\widetilde{R}(0)\pi^{+}+\hat{\pi}^{-}\widetilde{R}(0)\pi^{-}=
\hat{\pi}^{+}\widetilde{R}(0)\pi^{+}+\widetilde{R}(0)\pi^{-}.
\end{align*}
In particular,
\begin{equation*}
r(\infty)\Lambda^{0,1}=\bigl(\widetilde{R}(\infty)L^{+}\bigr)\wedge
\bigl(\hat{\pi}^{-}\widetilde{R}(\infty)S^{-}\bigr)=
\bigl(\widetilde{R}(\infty)L^{+}\bigr)\wedge \hat{S}^{-},
\end{equation*}
where, for the last equality, we note that
$\hat{\pi}^{-}\widetilde{R}(\infty)$ has kernel $S^{+}$ and so
injects (therefore surjects) when restricted to $S^{-}$.  Complex
conjugation now yields
$r(0)\Lambda^{1,0}=\hat{S}^{+}\wedge
\widetilde{R}(0)L^{-}$ and then \eqref{eq:72} gives
\begin{equation*}
\hL=(\hat{S}^{+}\wedge
\widetilde{R}(0)L^{-})\cap (\widetilde{R}(\infty)L^{+}\wedge
\hat{S}^{-})=
\widetilde{R}(\infty)L^{+}\wedge \widetilde{R}(0)L^{-}.
\end{equation*}
In particular, $\hat{L}^{+}=\widetilde{R}(\infty)L^{+}$ and
$\hat{L}^{-}=\widetilde{R}(0)L^{-}$ settling \eqref{eq:70}.

Finally, we consider $q$. From \eqref{eq:65}, we see that
$q$ preserves $S^{\pm}$ while $q^{1,0}_{|S^{-}}=0$ and similarly for
$\hat{q}$.  On the other hand, from \eqref{eq:73}, we have:
\begin{equation*}
\hat{q}^{1,0}=\Ad_{r(\infty)}q^{1,0}=\widetilde{r}(\infty)q^{1,0}\widetilde{r}(\infty)^{-1}.
\end{equation*}
Since $\widetilde{r}(\infty)_{|S^{+}}=\widetilde{R}(\infty)_{|S^{+}}$
with image $\hat{S}^{+}$, this immediately yields
\begin{equation*}
\hat{q}^{1,0}_{|\hat{S}^{+}}=\widetilde{R}(\infty)q^{1,0}\widetilde{R}(\infty)^{-1}_{|\hat{S}^{+}}.
\end{equation*}
On the other hand, both $\hat{q}^{0,1}$ and
$\widetilde{R}(\infty)q^{0,1}\widetilde{R}(\infty)^{-1}$ vanish on
$\hat{S}^{+}$ and so \eqref{eq:74} is established.
\end{proof}

We summarise the development in the following corollary:
\begin{corol}\label{th:26}
Let $R$ be an untwisted dressing gauge of a $k$-perturbed harmonic
$(S,\nabla^{\mu})$ in $\unC^6$ and define
$T_{0}\in\Gamma\rGL(\underline{\H}^{2})$ by
\begin{equation*}
T_0=\widetilde{R}(\infty)\pi^{+}+\widetilde{R}(0)\pi^{-}.
\end{equation*}

Then the dressing transform $\hS$ of $S$ by $R$ is given by
\begin{equation*}
\hat{S}=T_0ST_0^{-1}.
\end{equation*}

Moreover, if $S$ is the central sphere congruence of a constrained
Willmore surface $L$ with multiplier $q$, then the dressing transform
$\hat{L}$ and its
multiplier $\hat{q}$ are given by:
\begin{align*}
\hat{L}&=T_0L\\
\hat{q}&=T_0qT_{0}^{-1}.
\end{align*}
\end{corol}
\begin{proof}
The only thing to check here is that $T_0$ so defined is indeed an
isomorphism.  But the image of $T_0$ is
$\hat{S}^{+}+\hat{S}^{-}=\unC^4$ thanks to \eqref{eq:71}.
\end{proof}
Clearly there is some gauge freedom here: we could precompose $T_{0}$
with any gauge transformation which is a scalar multiple of the
identity on each of $S^{\pm}$.  We shall exploit this below.

\subsection{Darboux transforms and Riccati equations}
\label{sec:riccati-equation-1}

A transformation of Willmore surfaces in $S^{4}$ via solutions of a
Riccati equation is described in \cite[Chapter~12]{BurFerLesPed02}
while a related transform is derived by Leschke in \cite{Les11}.  We
now show that these transforms all amount to untwisted dressing by
simple factors as described in section~\ref{sec:untw-dress-gaug-1}.
Along the way, we extend the theory to constrained Willmore surfaces.

For this, we specialise the considerations of the last section to the
case where the dressing gauge $R$ is a simple factor $P_{\nu,M}$.
Recall that here $\nu\in\C^{\times}\setminus S^{1}$ and $M$ is a
$\nabla^{\nu}$-parallel null line subbundle with $M,\overline{M}$
complementary.  Equivalently, in the quaternionic formalism, $M=\Wedge^2W$
for $W$ a $\nabla^{\nu}$-parallel complex subbundle of $\unC^4$ such
that $\unC^4=W\oplus jW$ (so that $\mathrm{rank}W=2$).  If we now
define $\Gamma^W_{jW}(\mu)\in\Gamma\End(\unC^4)$ by
\begin{equation*}
\Gamma^W_{jW}(\mu)=
\begin{cases}
\mu&\text{on $W$}\\
1/\mu&\text{on $jW$}
\end{cases}
\end{equation*}
then
\begin{equation*}
\widetilde{P_{\nu,M}(\mu)}=
\Gamma^W_{jW}\biggl(\sqrt{\tfrac{(1-\nu^{*})(\mu-\nu)}{(1-\nu)(\mu-\nu^{*})}}\biggr).
\end{equation*}

We now have the following improvement on Corollary~\ref{th:26} in
case that $R$ is a simple factor:
\begin{thm}\label{th:27}
Let $(S,\nabla^{\mu})$ be $k$-perturbed harmonic in $\unC^6$. 
Fix $\nu\in\C^{\times}\setminus S^{1}$ and let $W\leq\C^4$ be a
$\nabla^{\nu}$-parallel subbundle such that $W\oplus jW=\unC^4$.
Define $\mu_0,\mu_1\in\C^{\times}$ by
\begin{equation}\label{eq:81}
\mu_0=i\frac{\nu+1}{\nu-1},\qquad 
\mu_1=i\frac{\nu^{*}+1}{\nu^{*}-1}.
\end{equation}
Further, define $X\in\Gamma\fgl(\underline{\H}^{2})$ by
\begin{equation*}
X=
\begin{cases}
\mu_0&\text{on $W$;}\\
\mu_1&\text{on $jW$.}
\end{cases}
\end{equation*}
and then $T\in\Gamma\fgl(\underline{\H}^{2})$ by
\begin{equation}\label{eq:77}
T=X-S.
\end{equation}
Then $P_{\nu,\Wedge^2W}$ is an untwisted dressing gauge for
$(S,\nabla^{\mu})$ on the open
set on which $T$ is an isomorphism and the dressing transform of 
$S$ by $P_{\nu,\Wedge^2W}$ is given by
\begin{subequations}
\label{eq:75}
\begin{equation}\label{eq:76}
\hS=TST^{-1}
\end{equation}
In particular, $\hS$ is $k$-perturbed harmonic.

Moreover, when $S$ is the central sphere congruence of a constrained
Willmore surface $L$ with multiplier $q$, then the untwisted dressing
transform of $L$ by $P_{\nu,\Wedge^2W}$ is given by
\begin{align}
\hat{L}&=TL\\
\hq&=TqT^{-1}.
\end{align}
\end{subequations}
\end{thm}
\begin{proof}
Set $R=P_{\nu,\Wedge^2W}$.  We claim that there are constants
$\lambda_0,\lambda_1\in\C^{\times}$ such that:
\begin{align}\label{eq:78}
\widetilde{R}(0)&=\lambda_0(X+i)&
\widetilde{R}(\infty)&=\lambda_1(X-i).
\end{align}
Given the claim, we immediately deduce that:
\begin{align*}
\widetilde{R}(0)_{|S^{-}}&=\lambda_0(X+i)_{|S^{-}}=\lambda_0T_{|S^{-}}\\
\widetilde{R}(\infty)_{|S^{+}}&=\lambda_1(X-i)_{|S^{+}}=\lambda_{1}T_{|S^{+}}.
\end{align*}
We note that $\mu_1=-\overline{\mu_0}$ so that $X$ and so $T$ are
$\fgl(\underline{\H}^2)$-valued. 
Moreover, $T$ is an isomorphism exactly when $\widetilde{R}(0)S^{-}\cap
\widetilde{R}(\infty)S^{+}=\set{0}$ or, equivalently, when
$R(0)V^{\perp}_{-}$ and $R(\infty)V^{\perp}_{+}$ are complementary,
that is, when $R$ is an untwisted dressing gauge.  The rest
of the theorem now follows at once from Proposition \ref{th:28}.

For the claim, note that both sides of each equation in \eqref{eq:78}
have $W$ and $jW$ as eigenspaces and so we must simply equate
eigenvalues and require that
\begin{align*}
\sqrt{\frac{\nu(1-\nu^{*})}{\nu^{*}(1-\nu)}}&=\lambda_0(\mu_0+i)&
\sqrt{\frac{1-\nu^{*}}{1-\nu}}&=\lambda_1(\mu_0-i)\\
\sqrt{\frac{\nu^{*}(1-\nu)}{\nu(1-\nu^{*})}}&=\lambda_0(\mu_1+i)&
\sqrt{\frac{1-\nu}{1-\nu^{*}}}&=\lambda_1(\mu_i-i).
\end{align*}
These amount to linear equations for the $\mu_i$ which are solved by
\eqref{eq:81} and then the $\lambda_i$ are given by
\begin{align*}
\lambda_0&=-\frac{i}{2}\sqrt{\frac{(\nu-1)(\nu^{*}-1)}{\nu\nu^{*}}}&
\lambda_1&=-\frac{i}{2}\sqrt{(\nu-1)(\nu^{*}-1)}.
\end{align*}
\end{proof}
\begin{rem}
For Willmore $L$ and, more generally, harmonic $S$, the construction
in Theorem~\ref{th:27} of $\hS$ and $\hat{L}$ from a
$\nabla^{\nu}$-parallel $W$ coincides with Leschke's \emph{$\nu$-Darboux transform}
\cite[Theorems~4.2 and 4.4]{Les11}.  She conjectured that this
procedure should amount to a dressing transform as we have just
affirmed.
\end{rem}

Theorem~\ref{th:27} offers a different perspective on untwisted
dressing by simple factors by focussing on the field of endomorphisms
$X$ rather than the dressing gauge $P_{\nu,\Wedge^2W}$.  The key
conditions on $X$, that it have constant eigenvalues and
$\nabla^{\nu}$ or $\nabla^{\nu^{*}}$-parallel eigenspaces $W, jW$, are
encapsulated in a Riccati equation with a conserved quantity:

\begin{prop}
\label{th:29}
Let $(S,\nabla^{\mu})$ be $k$-perturbed harmonic,
$X\in\Gamma\fgl(\underline{\H}^2)$ and $\mu_0\in\C\setminus\R$.  Then
$X$ has eigenvalue $\mu_0$ with rank $2$ $\nabla^{\nu}$-parallel eigenbundle
if and only if 
\begin{subequations}
\label{eq:82}
\begin{gather}
\label{eq:83}
(\mu_1-\mu_{0})dX=X(\beta^{*}-\beta)X+(\mu_0\beta-\mu_1\beta^{*})X+X(\mu_1\beta-\mu_0\beta^{*})
+\mu_0\mu_1(\beta^{*}-\beta)\\
\label{eq:84}
X^2-(\mu_0+\mu_1)X+\mu_0\mu_1=0,
\end{gather}
where $\mu_1=\overline{\mu_0}$, $\nabla^{\nu}=d+\beta$ and $\nabla^{\nu^{*}}=d+\beta^{*}$.
\end{subequations}
\end{prop}
\begin{proof}
First note that if $X$ has rank $2$ $\mu_0$-eigenbundle $W$ then it
also has $\mu_1$-eigenbundle $jW$ with $\unC^4=W\oplus jW$ so that
the minimal polynomial of $X$ is $X^2-(\mu_0+\mu_1)X+\mu_0\mu_1$.
Conversely, if $X$ solves \eqref{eq:84} with $\mu_0$ non-real, then,
since $[X,j]=0$, $X$ is not a scalar multiple of the identity and so
has $X^2-(\mu_0+\mu_1)X+\mu_0\mu_1$ as its minimal polynomial.  It
follows at once that $X$ has eigenvalues $\mu_0,\mu_1$ with rank $2$
eigenbundles.

Thus we may assume that $X$ has rank $2$ $\mu_0$-eigenbundle $W$ and
the only remaining issue is whether $W$ is $\nabla^{\nu}$-parallel, or
equivalently, $jW$ is $\nabla^{\nu^{*}}$-parallel.  This is easily
seen to be equivalent to the demand that
\begin{align*}
(X-\mu_0)\nabla^{\nu}(X-\mu_1)&=0\\
(X-\mu_1)\nabla^{\nu^{*}}(X-\mu_0)&=0,
\end{align*}
or, since these last have image in different eigenbundles,
\begin{equation*}
(X-\mu_0)\nabla^{\nu}(X-\mu_1)=
(X-\mu_1)\nabla^{\nu^{*}}(X-\mu_0).
\end{equation*}
However, writing this out in terms of $\beta$ and $\beta^{*}$
promptly yields \eqref{eq:83}.
\end{proof}
\begin{rem}
In the situation of Theorem \ref{th:27} with $\mu_0$ given by
\eqref{eq:81}, we see that the excluded case that $\mu_0\in\R$
amounts to the already excluded case of $\nu\in S^{1}$.
\end{rem}

Observe that the coefficients in the Riccati equation \eqref{eq:83}
are all pure imaginary so that solutions with initial condition in
$\fgl(\underline{\H}^{2})$ remain in $\fgl(\underline{\H}^{2})$.  In
fact, much more is true: the Riccati equation is completely
integrable and admits \eqref{eq:84} as a first integral.  All this
uses almost nothing about the specifics of the situation beyond the
flatness of the connections $\nabla^{\nu}$ and $\nabla^{\nu^{*}}$.
Indeed, we have:
\begin{prop}
Let $\mathsf{A}$ be a complex finite-dimensional associative algebra
with unit, $\mu_0\neq\mu_1\in\C$ and $\beta,\beta^{*}$
$\mathsf{A}$-valued $1$-forms on $\Sigma$ defining flat connections:
\begin{equation*}
d\beta+\beta\wedge\beta=d\beta^{*}+\beta^{*}\wedge\beta^{*}=0.
\end{equation*}
Then the Riccati equation \eqref{eq:83}
\begin{equation*}
(\mu_1-\mu_{0})dX=X(\beta^{*}-\beta)X+(\mu_0\beta-\mu_1\beta^{*})X+X(\mu_1\beta-\mu_0\beta^{*})
+\mu_0\mu_1(\beta^{*}-\beta)
\end{equation*}
for $X:\Sigma\to\mathsf{A}$ is completely integrable.

Moreover, if $X$ is a solution and $I=X^2-(\mu_0+\mu_1)X+\mu_0\mu_1$,
then $I$ solves the linear equation
\begin{equation*}
(\mu_1-\mu_0)dI=I(\beta^{*}-\beta)X+X(\beta^{*}-\beta)I+I(\mu_1\beta-\mu_0\beta^{*})+
(\mu_0\beta-\mu_1\beta^{*})I
\end{equation*}
and so, in particular, vanishes identically if it vanishes at a single point.
\end{prop}
\begin{proof}
One can prove complete integrability by direct computation but we
offer a more conceptual and, perhaps, instructive proof based on the
well known classical observation that Riccati equations are the affine
expression of linear equations in homogeneous coordinates.  So
consider a general Riccati equation
\begin{equation*}
dX= XAX+BX+XC+D,
\end{equation*}
with $\mathsf{A}$-valued $1$-forms $A,B,C,D$.  Write $X=PQ^{-1}$ for
$P,Q:\Sigma\to\mathsf{A}$ (we can always do this near some initial
point $p_0\in\Sigma$ by taking $Q(p_0)=1$).  Then the equation
becomes
\begin{equation*}
dP-BP-DQ-PQ^{-1}(dQ+AP+CQ)=0
\end{equation*}
so that integrability is guaranteed by flatness of the connection on
the trivial $\mathsf{A}\oplus\mathsf{A}$ bundle given by
\begin{equation*}
d+
\begin{pmatrix}
-B&-D\\A&C
\end{pmatrix}
.
\end{equation*}
In the case at hand, this connection reads
\begin{equation*}
d+\frac{1}{\mu_1-\mu_0}
\begin{pmatrix}
\mu_1\beta^{*}-\mu_0\beta&\mu_0\mu_1(\beta-\beta^{*})\\
\beta^{*}-\beta&\mu_1\beta-\mu_0\beta^{*}
\end{pmatrix}=
\begin{pmatrix}
\mu_0&\mu_1\\1&1
\end{pmatrix}
\cdot\left(d+
  \begin{pmatrix}
  \beta&0\\0&\beta^{*}
  \end{pmatrix}
\right)
\end{equation*}
which is clearly flat.

The derivation of the linear equation for $I$ is a straightforward
computation which we leave to the interested reader.
\end{proof}

Let us summarise this development which amounts to a construction of
new $k$-perturbed harmonic bundles from old via solving a Riccati
equation:
\begin{corol}\label{th:30}
Let $(S,\nabla^{\mu})$ be $k$-perturbed harmonic in $\unC^{6}$ and
fix $\nu\in\C^{\times}\setminus S^1$.  Define
$\mu_0,\mu_1\in\C^{\times}$ by \eqref{eq:81}.  Fix a base-point
$x_o\in\Sigma$ and $X_o\in\fgl(\H^2)$ satisfying
$X_o^2-(\mu_0+\mu_1)X_o+\mu_0\mu_1=0$.   

There is (locally) $X\in\Gamma\fgl(\underline{\H}^2)$ solving:
\begin{gather*}
(\mu_1-\mu_{0})dX=X(\beta^{*}-\beta)X+(\mu_0\beta-\mu_1\beta^{*})X+X(\mu_1\beta-\mu_0\beta^{*})
+\mu_0\mu_1(\beta^{*}-\beta)\\
X(x_o)=X_o,
\end{gather*}
where $\nabla^{\nu}=d+\beta$ and $\nabla^{\mu}=d+\beta^{*}$.
Then, with $T=X-S$, $\hat{S}=TST^{-1}$ is $k$-perturbed harmonic on the
open set on which $T$ is invertible.

Moreover, if $S$ is the central sphere congruence of a constrained
Willmore surface $L$, then $\hat{S}$ is the central sphere
congruence of a constrained Willmore surface $TL$.
\end{corol}

Of course, our Riccati equation for $X$ can be phrased as a Riccati
equation for $T$ the form of which, for general $\nu$, is not very
edifying.  However, when $\nu$ is real, matters simplify considerably
and we not only recover the Darboux transforms of Willmore surfaces
described in \cite[Section~12.3]{BurFerLesPed02} but generalise them
to the constrained Willmore case.

So let $L$ be a constrained Willmore surface with multiplier $q$ and
central sphere congruence $S$.  We then have
$\nabla^{\mu}=d+(\mu-1)(\cA_S+q)^{1,0}+(\mu^{-1}-1)(\cA_S+q)^{0,1}$.
Fix $\nu\in\R\setminus\set{\pm 1}$ so that $\nu^{*}=1/\nu$ and then
\begin{equation*}
\mu_0=i \frac{\nu+1}{\nu-1}=-\mu_1.
\end{equation*}
Set 
\begin{equation}
\label{eq:80}
\rho=-\frac{(\nu-1)^2}{4\nu}=\frac{2-\nu-\nu^{*}}{4}
\end{equation}
so that
\begin{equation*}
\rho^{-1}-1=\mu_0^2=\mu_1^2.
\end{equation*}
We may now use \eqref{eq:56} along with \eqref{eq:68} and
\eqref{eq:65} to write the Riccati equation \eqref{eq:83} for $X$ as
the following Riccati equation for $T=X-S$:
\begin{equation}
\label{eq:85}
dT=\rho T\bigl(2*(\cA_S+q)\bigr)T-4\rho Tq-2*(\cQ_S+q)
\end{equation}
with first integral $(T+S)^2-(\rho^{-1}-1)$.

In this situation, $\rho^{-1}<1$ and, conversely, if $\rho^{-1}<1$,
we can rearrange \eqref{eq:80} and recover $\nu,\nu^{*}$ as the two
real roots of $\mu^2+2(2\rho-1)\mu+1=0$.  We therefore have:
\begin{thm}
\label{th:31}
Let $L$ be a constrained Willmore surface in $S^4$ with multiplier
$q$ and central sphere congruence $S$.  Let $\rho\in\R$ with
$\rho^{-1}<1$ and let $T\in\fgl(\underline{\H}^2)$ be a solution of
the integrable Riccati equation \eqref{eq:85} with
$(T+S)^2-(\rho^{-1}-1)=0$ at one, and hence every, point.

Then $\hat{L}=TL$ is also a constrained Willmore surface with
multiplier $TqT^{-1}$ and central sphere congruence $TST^{-1}$ on the
open set on which $T$ is invertible.

Moreover, $\hat{L}$ is an untwisted dressing transform of $L$ by a
simple factor.
\end{thm}

For Willmore surfaces (thus $q=0$), this construction of $\hat{L}$ is
the Darboux transform of \cite[Section~12.3]{BurFerLesPed02} with the
caveat that their Riccati equation swops the roles of $\cA_S$ and
$\cQ_S$ and their $T$ is the inverse of ours.  This is, of course, a
well-known symmetry of the Riccati equation.

%\bibliography{master}

\end{document}